\documentclass{amsart}
\usepackage{latexsym,amsxtra,amscd,ifthen}
\usepackage{amsfonts}
\usepackage{verbatim}
\usepackage{amsmath}
\usepackage{amsthm}
\usepackage{amssymb}

\usepackage{enumerate}

\numberwithin{equation}{subsection}

\theoremstyle{plain}
\newtheorem*{theorem}{Theorem}
\newtheorem*{lemma}{Lemma}

\newtheorem*{proposition}{Proposition}
\newtheorem*{corollary}{Corollary}

\newcommand*{\inth}{\textstyle \int}
\theoremstyle{definition}
\newtheorem*{definition}{Definition}
\newtheorem*{example}{Example}
\newtheorem*{examples}{Examples}

\newtheorem*{remark}{Remark}

\newtheorem*{hypotheses}{Hypotheses}
\newtheorem*{remarks}{Remarks}
\newtheorem*{remarksA}{Remarks A}
\newtheorem*{remarksB}{Remarks B}
\newtheorem*{question}{Question}



\newcommand{\rank}{\operatorname{rank}}

 \DeclareMathOperator{\gldim}{gldim}

\DeclareMathOperator{\Ext}{Ext} 
\DeclareMathOperator{\prdim}{projdim}

 \DeclareMathOperator{\ann}{ann}

 \DeclareMathOperator{\Aut}{Aut}
 
\DeclareMathOperator{\injdim}{injdim}

\DeclareMathOperator{\GKdim}{GKdim} \DeclareMathOperator{\End}{End}
 
\DeclareMathOperator{\Hom}{Hom}

\newcommand{\C}{\mathbb{C}}

\DeclareMathOperator{\PIdeg}{{\mathrm{PI}-\mathrm{deg}}}

\begin{document}

\title[Hopf algebras of GK-dimension One]
{Prime regular Hopf Algebras of GK-dimension One}

\author{K.A. Brown and J.J. Zhang}

\address{Brown: Department of Mathematics,
University of Glasgow, Glasgow G12 8QW, UK}

\email{kab@maths.gla.ac.uk}

\address{zhang: Department of Mathematics, Box 354350,
University of Washington, Seattle, Washington 98195, USA}

\email{zhang@math.washington.edu}

\begin{abstract}
This paper constitutes the first part of a program to classify all
affine prime regular Hopf algebras $H$ of Gelfand-Kirillov dimension
one over an algebraically closed field of characteristic zero. We
prove a number of properties of such an algebra, list some classes
of examples, and then prove that - when the PI-degree of $H$ is
prime - our list contains all such algebras.
\end{abstract}

\subjclass[2000]{Primary 16E65, 16W30, 16P40; Secondary 16S30,
16S34, 16W35, 16R99 }


\keywords{Hopf algebra, Gelfand-Kirillov dimension, Artin-Schelter
regular, homological integral}


\maketitle


\setcounter{section}{-1}
\section{Introduction}
\label{yysec0}
\subsection{}
\label{yysec0.1}
Recent years have seen substantial progress in our
understanding of (infinite dimensional) noetherian Hopf algebras
\cite{BG1,LWZ,WZ1,BZ}. One noteworthy aspect of this development has
been the increasing role of homological algebra - for example, it
was shown in \cite{WZ1} that every affine noetherian Hopf algebra
satisfying a polynomial identity is Auslander-Gorenstein and
Cohen-Macaulay, and hence, by \cite[Lemma 6.1]{BZ}, is
Artin-Schelter (AS) Gorenstein; see (\ref{yysec1.2}) for the
definition of the latter condition.

Recall that, by fundamental results of Small, Stafford and Warfield
\cite{SSW}, a semiprime affine algebra of GK-dimension one is a
finite module over its center, which is affine and hence noetherian.
Given the homological developments mentioned above, it seems to us a
feasible target to determine \emph{all} prime affine Hopf algebras
of GK-dimension one. In this paper we begin the campaign by
considering the case where
\begin{quote} $H$ \textit{ is an affine noetherian Hopf algebra over
an algebraically closed field } $k$ \textit{ of characteristic 0,
and } $H$ \textit{ has GK-dimension one and is prime and regular.}
\end{quote}
Throughout, we'll say that $H$ satisfies hypotheses
($\mathbf{H}$) when the displayed conditions are in force. Here,
\textit{regular } means that $H$ has finite global
dimension, say $d$. In fact the Cohen-Macaulay property of $H$
forces $d=1$ here, so the objects of study are hereditary noetherian
prime rings.

\subsection{Examples}
\label{yysec0.2}
By a basic result in the theory of
algebraic groups \cite[Theorem 20.5]{Hu}, there are two
\emph{commutative} Hopf $k-$algebras which are domains of
GK-dimension one, namely $k[x]$ and $k[x^{\pm 1}]$, the coordinate
rings of $(k,+)$ and $(k\setminus \{0\},\times)$ respectively. Since
these groups are abelian, their coordinate rings are cocommutative;
if one retains the cocommutativity hypothesis, but drops
commutativity, then the only additional algebras satisfying
($\mathbf{H}$) which enter the arena are the group algebras
$k{\mathbb D}$ of the infinite dihedral group (\ref{yysec3.2}), and
the algebras which we call here the \emph{cocommutative Taft algebras}.
These, discussed in (\ref{yysec3.3}), are skew group algebras of a finite
cyclic group $\langle g \rangle$ of order $n$, acting on $k[x]$,
with $gxg^{-1} = \xi x,$ where $\xi$ is a primitive $n$th root of
one in $k.$

For each $n > 1$ there is a finite collection of
\emph{noncocommutative} Taft algebras $H(n,t,\xi)$, also defined in
(\ref{yysec3.3}), for $\xi$ a primitive $n$th root of one in $k.$
(Here, $k$ can be any field containing a primitive $n$th root of
unity.) The algebras $H(n,t,\xi)$ are well known, being factor
algebras of a Borel subalgebra of $U_{\xi}(\mathfrak{sl}(2,k)).$

The above are \emph{not}, however, the only examples of algebras
satisfying ($\mathbf{H}$). A further class of examples was
discovered by Liu \cite{Liu} in 2006: these are built by letting a
skew primitive element $y$ ``of finite index of nilpotency $n$'' act
on $k[x^{\pm 1}]$ by setting $yx = \xi xy$, and forming a
``non-split'' extension. Liu's algebras are recovered as a special
case of a larger class of algebras in (\ref{yysec3.4}), where we show
that his construction can be extended to include Hopf algebras built
over the group algebra $k(C_{\infty} \times C_b),$ where $b$ is an
arbitrary divisor of $n.$

In all the examples listed in this subsection, the PI-degree of
the algebra is $n$, (where, of course, $n=1$ for the commutative
algebras and $n=2$ for $k{\mathbb D}$).

\subsection{Prime PI-degree}
\label{yysec0.3}
It may be that the only algebras $H$
satisfying ($\mathbf{H}$) are those listed in (\ref{yysec0.2}),
but we cannot prove this at present. However it is a corollary of our
main result that
$$\begin{aligned}  \textit{if } H \textit{ satisfies }
(\mathbf{H}) \textit{ and the PI-degree }
 n \textit{ of } H \textit{ is prime, then } H \textit{ occurs in
(\ref{yysec0.2}).}
\end{aligned}$$
To give a more precise statement of our main result and some idea
of its proof, we must introduce some homological algebra.

\subsection{Integral order and integral minor}
\label{yysec0.4}
A fundamental tool for the study of algebras satisfying ($\mathbf{H}$)
is the \emph{homological integral} $\int^l_H$ defined in \cite{LWZ}.
This is the one-dimensional $k-$vector space and $H-$bimodule
$\Ext^1_H (k,H),$ where $k$ here denotes the trivial left
$H-$module. Let $\pi : H \rightarrow k$ be the algebra map given by
the \emph{right} structure of $\int^l_H$. As we recall in
(\ref{yysec2.1}), $\pi$ determines groups $G^l_{\pi}$ and
$G^r_{\pi}$ of left and right \emph{winding automorphisms} of $H$.
In our setting, and indeed for any affine Hopf algebra which is a
finite module over its center,
$$ n := |G^l_{\pi}| = |G^r_{\pi}| < \infty, $$
(Theorem \ref{yysec2.3}(c)); the integer $n$ is called the \emph{
integral order} of $H$, denoted $io(H).$

When $H$ is commutative, $\pi$ is - clearly - the counit $\epsilon,$
and so $io(H)=1$. So, speaking crudely, $io(H)$ is a measure of the
noncommutativity of $H$; and indeed it will become apparent from our
results that, for $H$ satisfying ($\mathbf{H}$), $io(H)=1$ only if
$H$ is commutative. Notice also that, when $H$ is cocommutative,
$G^l_{\pi}=G^r_{\pi}$; so we can encode the ``degree of
cocommutativity'' of $H$ by means of the index \begin{eqnarray}
im(H) :=|G^l_{\pi} : G^l_{\pi} \cap G^r_{\pi}|,
\label{greet}\end{eqnarray} which we call the \emph{integral minor}
of $H$, (\ref{yysec2.4}). Thus, when $H$ is cocommutative $im(H)=1,$
and, paralleling the corresponding statement for $io(H),$ it will
follow from results in \S 4 that, when $H$ satisfies ($\mathbf{H}$),
$im(H)=1$ only if $H$ is cocommutative.

The connection between the integral order and noncommutativity is in
fact more far-reaching than indicated above. By results of
\cite{LWZ} recalled in \S 2, for \emph{any} prime regular affine PI
Hopf $k-$algebra, the integral order is a lower bound for the
PI-degree; and
\begin{eqnarray} \textit{when } H \textit{ satisfies
}(\mathbf{H}), \quad io(H) =
\PIdeg(H).\label{ole}\end{eqnarray} In other words, for
algebras $H$ satisfying ($\mathbf{H}$) ``the non-commutativity of
the algebra is entirely encoded by the integral $\int^l_H$.''

\subsection{Main theorem}
\label{yysec0.5}
We can now state the main result of this paper:

\begin{theorem} Let $H$ be a Hopf algebra satisfying ($\mathbf{H}$), and
suppose that the integral minor $im(H)$ is either 1 or $io(H)$. Then
$H$ is one of the algebras listed in (\ref{yysec0.2}).
\end{theorem}

In view of (\ref{ole}) and (\ref{greet}), the corollary stated in
(\ref{yysec0.3}) follows.

It's worth noting that all the algebras appearing in
(\ref{yysec0.2}) are \emph{pointed}. We do not know whether this is
true for \emph{all} prime affine noetherian Hopf $k$-algebras of
GK-dimension one, nor whether there is a direct way to prove
pointedness, without proceeding first to list all the algebras.

The first steps in the proof of the theorem are valid in a much more
general setting: we show in \S 2, building mainly on work from
\cite{LWZ}, that if a Hopf algebra $A$ has finite integral order
$n$, where $n$ is a unit in $k,$ then $A$ is strongly graded by the
cyclic group of order $n$, in two different ways, with identity
components the invariant algebras $A^{G^l_{\pi}}$ and
$A^{G^r_{\pi}}.$ When $A$ is in addition regular affine semiprime
and has GK-dimension one, these two invariant algebras, and their
intersection, are affine Dedekind domains (Theorem \ref{yysec2.5}).

\S 4 covers the easy case of the main theorem, namely the case where
the Hopf algebra $H$ satisfies ($\mathbf{H}$) and $im(H)=1.$ Under
these hypotheses, $H_0 := H^{G^l_{\pi}}=H^{G^r_{\pi}}$ and this
Dedekind domain is a Hopf subalgebra of $H$. At this point a pivotal
dichotomy comes into play: $H_0$ must either be $k[x]$ - the
\emph{primitive case}; or $H_0$ is $k[x^{\pm 1}]$ - the
\emph{group-like case}. When $im(H)=1$ it's not hard to show that in
the primitive case $H$ is a cocommutative Taft algebra, and in the
group-like case it's the group algebra of the dihedral group.

The proof of the case $im(H)=io(H)$ of the main theorem occupies \S5
and \S 6. In \S 5 an analysis of algebras strongly graded by a
finite abelian group, with commutative component of degree 0,
culminates in the conclusion that, if $H$ satisfies ($\mathbf{H}$)
and $io(H) = im(H)$, then the Dedekind domain $H_0 := H^{G^l_{\pi}}
\cap H^{G^r_{\pi}}$ is the center of $H$ and is a Hopf subalgebra,
whence it is either $k[x]$ or $k[x^{\pm 1}].$ The dichotomy
mentioned in the previous paragraph therefore features again, and in
\S 6 this is exploited to show that only the algebras of \S 3 can
occur. This is achieved by first studying the finite dimensional
Hopf algebra $H/ (\ker \epsilon \cap H_0)H$, which we call the
\emph{twistor} of $H$, and then lifting its structure to describe
$H$ itself.

Finally, in \S 7, we discuss a number of questions and conjectures
arising from this work.

\bigskip

\section{General preparations}
\label{yysec1}

\subsection{Standard notation}
\label{yysec1.1} Throughout, $k$ will denote an algebraically closed
field. Unless stated otherwise all vector spaces are over $k$, and
an algebra or a ring $A$ always means a $k$-algebra with associative
multiplication $m_A$ and with unit 1. All $A$-modules will be by
default left modules. Let $A^{\sf op}$ denote the opposite algebra
of $A$.

When $A$ is a Hopf algebra we shall use the symbols $\Delta$,
$\epsilon$ and $S$ respectively for its coproduct, counit and
antipode. The coproduct of $a \in A$ will be denoted by
$\Delta(a)=\sum a_1 \otimes a_2$. The dual Hopf algebra of $A$ will
be denoted by $A^{\circ}$. For background regarding Hopf algebras,
see for example \cite{Mo}.

\subsection{Artin-Schelter condition}
\label{yysec1.2} While our main interest is in Hopf algebras, the
following two key definitions apply more generally to any noetherian
\emph{augmented $k$-algebra} - that is, a noetherian $k-$algebra $A$
with a fixed augmentation $\epsilon : A \longrightarrow k$.

\begin{definition}
Let $(A, \epsilon)$ be an augmented noetherian algebra.
\begin{enumerate}
\item
We shall say $A$ has {\it finite injective dimension} if the
injective dimensions of $_AA$ and $A_A$, $\injdim _AA$ and
$\injdim A_A$, are both finite. In this case these integers
are equal by \cite{Za}, and we write $d$ for the common value.
We say $A$ is {\it regular} if it has finite global dimension,
$\gldim_AA < \infty$. Right global dimension always equals
left global dimension \cite[Exercise 4.1.1]{We}; and, when
finite, the global dimension equals the injective dimension.
\item
Write $k$ for the trivial $A$-module $A/\ker \epsilon$. Then
$A$ is {\it Artin-Schelter Gorenstein}, which we usually
abbreviate to {\it AS-Gorenstein}, if
\begin{enumerate}
\item[(AS1)]
$\injdim {_AA}=d<\infty$,
\item[(AS2)]
$\dim_k \Ext^d_A({_Ak},{_AA})=1$ and $\Ext^i_A({_Ak},{_AA})=0$
for all $i\neq d$,
\item[(AS3)]
the right $A$-module versions of (AS1,AS2) hold.
\end{enumerate}
\item
If, further, $\gldim A=d$, then $A$ is called {\it
Artin-Schelter regular}, usually shortened to {\it AS-regular}.
\end{enumerate}
\end{definition}

\subsection{Homological integrals}
\label{yysec1.3}
Here is the natural extension to augmented algebras of a
definition recently given in \cite{LWZ} for Hopf algebras. This
definition generalises a familiar concept from the case of a
finite dimensional Hopf algebra \cite[Definition 2.1.1]{Mo}).

\begin{definition}
Let $(A, \epsilon)$ be a noetherian augmented algebra, and suppose
$A$ is AS-Gorenstein of injective dimension $d$. Any nonzero
element of the one-dimensional $A-$bimodule $\Ext^d_A({_Ak},{_AA})$
is called a {\it left homological integral} of $A$. We denote
$\Ext^d_A({_Ak},{_AA})$ by $\inth^l_A$. Any nonzero element in
$\Ext^d_{A^{\sf op}}({k_A},{A_A})$ is called a {\it right
homological integral} of $A$. We write $\inth^{r}_A=
\Ext^d_{A^{\sf op}}({k_A},{A_A})$. Abusing language slightly, we
shall also call $\inth^l_A$ and $\inth^r_A$ the left and the right
homological integrals of $A$ respectively. When no confusion as
to the algebra in question seems likely, we'll simply write
$\inth^l$ and $\inth^r$ respectively.
\end{definition}

\section{Classical Components}
\label{yysec2}
Although our primary goal is to classify affine prime
regular Hopf $k$-algebras of GK-dimension one, we give in this
section some preliminary results in a more general setting than is
strictly needed for this classification, in the hope that some of
these ideas will prove useful in future work. Thus, throughout \S
\ref{yysec2}, $A$ will be an arbitrary Hopf $k$-algebra.

\subsection{Hopf subalgebras associated to winding automorphisms}
\label{yysec2.1}
Given an algebra homomorphism $\pi: A\to k$, we write $\Xi^l_\pi$
for the left winding automorphism of $A$ associated to $\pi$, namely,
$$\Xi^l_{\pi}(a)=\sum \pi(a_1)a_2$$
for all $a\in A;$ and we use $\Xi^r_{\pi}$ for the right winding
automorphism of $A$ associated to $\pi,$ defined for $a \in A$ by
$$\Xi^r_{\pi}(a)=\sum a_1\pi(a_2).$$
Fix now a set $\Pi$ of algebra homomorphisms $\pi$ from $A$ to $k$.
Let $G^l_{\Pi}$ [resp. $G^r_{\Pi}$] be the subgroup of
$\Aut_{k-\mathrm{alg}}(A)$ generated by $\{\Xi^l_{\pi}: \pi \in \Pi \}$
[resp. $\{\Xi^r_{\pi}: \pi \in \Pi \}$]. Set
$G_{\Pi} := \langle G^l_{\Pi}, G^r_{\Pi} \rangle \subseteq
\Aut_{k-\mathrm{alg}}(A)$. (In applications, $A$ will be
AS-Gorenstein and we will take $\Pi=\{\pi\}$ with $\pi$ the map
$A\to A/\mathrm{r}.\ann(\int^l)$.)

Let $A_{\Pi}$ [resp. $A^l_{\Pi}$, $A^r_{\Pi}$] be the subalgebra of
invariants $A^{G_{\Pi}}$ [resp. $A^{G^l_{\Pi}}$, $A^{G^r_{\Pi}}$].
When $\Pi$ consists of a single map, $\Pi = \{ \pi \},$ we'll write
$A_{\pi}$, $A^l_{\pi}$, $A^r_{\pi}$,  and $G_{\pi}$, $G^l_{\pi}$,
$G^r_{\pi}$, for these subalgebras of $A$ and subgroups of
$\mathrm{Aut}_{k-\mathrm{alg}}(A)$.

The action of $G^l_{\Pi}$ is the restriction to the group-like
elements $\langle \Pi \rangle \subseteq A^\circ$ of the \emph{right}
hit action of $A^\circ$ on $A$; and correspondingly for $G^r_{\Pi}$
(see e.g. \cite[1.6.5 and 1.6.6]{Mo}).

\begin{proposition}
Retain the above notation, and let $\pi,\pi'\in \Pi$.
\begin{enumerate}
\item
$A^l_{\Pi}$, $A^r_{\Pi}$ and $A_{\Pi}$ are subalgebras of
$A$ and $A_{\Pi}=A^l_{\Pi}\cap A^r_{\Pi}$.
\item
If $A$ is cocommutative, then $A^l_{\Pi}=A^r_{\Pi}=A_{\Pi}$.
\item
$(\Xi^l_{\pi}\otimes Id_A)\Delta=\Delta \Xi^l_{\pi}$ and
$(Id_A\otimes \Xi^r_{\pi})\Delta=\Delta \Xi^r_{\pi}$. Thus
$A^l_{\Pi}$ is a right coideal subalgebra of $A$ and $A^r_{\Pi}$
is a left coideal subalgebra of $A$.
\item
$\Xi^l_{\pi}\; \Xi^r_{\pi'}=\Xi^r_{\pi'}\; \Xi^l_{\pi}$. In
particular, $G_{\Pi}$ is abelian if and only if $G^l_{\Pi}$ and
$G^r_{\Pi}$ are abelian.
\item
$\Xi^r_{\pi}\; S=S\; (\Xi^l_{\pi})^{-1}$. Consequently, $S(A^l_{\Pi})
\subseteq A^r_{\Pi}$ and $S(A^r_{\Pi})\subseteq A^l_{\Pi}$.
\item
$(Id_A\otimes \Xi^l_{\pi})\Delta=(\Xi^r_{\pi}\otimes Id_A) \Delta$.
Thus $\Delta(a)\in A\otimes A^l_{\Pi}$ if and only if
$\Delta(a)\in A^r_{\Pi}\otimes A$ if and only if $\Delta(a)\in
A^r_{\Pi}\otimes A^l_{\Pi}$.
\item
Let $\widetilde{A_{\Pi}}:=\{a\in A\;|\; \Delta(a)\in A_{\Pi}
\otimes A_{\Pi}\}$. Then $\widetilde{A_{\Pi}}$ is also equal to
$$
\{a\in A\;|\; \Delta(a)\in A^l_{\Pi}\otimes A^l_{\Pi}\},$$
$$\{a\in A\;|\; \Delta(a)\in A^r_{\Pi}\otimes A^r_{\Pi}\},$$
and
$$\{a\in A\;|\; \Delta(a)\in A^r_{\Pi}\otimes A^l_{\Pi}\}.$$
\item
$\widetilde{A_{\Pi}}$ is a Hopf subalgebra of $A$. In fact,
$\widetilde{A_{\Pi}}$ is the largest Hopf subalgebra of
$A^l_{\Pi}$, of $A^r_{\Pi}$ and of $A_{\Pi}$. If $A^l_{\Pi}=
A^r_{\Pi}$ (and so in particular if $A$ is cocommutative), then
$\widetilde{A_{\Pi}}=A_{\Pi}$.
\end{enumerate}
\end{proposition}

\begin{proof} (a,b) Clear.

(c) Let $\mu$ denote the multiplication map $A \otimes A
\longrightarrow A.$ Since $\Xi_{\pi}^l=\mu(\pi\otimes Id_A)\Delta$,
$$
\begin{aligned}
(\Xi_{\pi}^l\otimes Id_A)\Delta
&=(\mu \otimes Id_A)(\pi\otimes Id_A\otimes Id_A)(\Delta\otimes Id_A)\Delta\\
&=(\mu \otimes Id_A)(\pi\otimes Id_A\otimes Id_A)(Id_A\otimes \Delta)\Delta\\
&=(\mu \otimes Id_A)(Id_k\otimes \Delta) (\pi\otimes Id_A)\Delta\\
&=\Delta \Xi_{\pi}^l.
\end{aligned}
$$
Similarly, $(Id_A\otimes \Xi_{\pi}^r)\Delta=\Delta\Xi_{\pi}^r$.
Let $a\in A^l_{\Pi}$. Then, for all $\pi \in \Pi,$
$\Xi_{\pi}^l(a)=a,$ and hence
$$(\Xi_{\pi}^l\otimes Id_A)\Delta(a)=\Delta \Xi_{\pi}^l(a)
=\Delta(a).$$ Thus $\Delta(a)\in A^l_{\Pi}\otimes A$ and hence
$\Delta(A^l_{\Pi})\subset  A^l_{\Pi}\otimes A$. Similarly,
$\Delta(A^r_{\Pi})\subset  A\otimes A^r_{\Pi}$.

(d) Clear.

(e) Note that $(\Xi_{\pi}^l)^{-1}(a)=\sum \pi S(a_1)a_2$ for all
$a\in A$. Hence, using \cite[1.5.10]{Mo},
$$\Xi_{\pi}^r S(a)=\sum S(a_2)\pi S(a_1)=S(\sum \pi S(a_1)a_2)=
S(\Xi_{\pi}^l)^{-1}(a).$$ Thus $S(A^l_{\Pi})\subseteq A^r_{\Pi}$.
Similarly, $S(A^r_{\Pi})\subseteq A^l_{\Pi}$.

(f) For $a\in A$,
$$(Id_A\otimes \Xi^l_{\pi})\Delta(a)=
\sum a_1 \otimes \pi(a_2)a_3= \sum a_1\pi(a_2)\otimes a_3=
=(\Xi^r_{\pi}\otimes Id_A) \Delta(a),$$ proving the first
statement. Now $\Delta(a) \in A\otimes A^l_{\Pi}$ if and only if
$(Id_A\otimes \Xi^l_{\pi})\Delta(a)=\Delta(a)$ for all $\pi \in
\Pi$. By the first assertion, this happens if and only if
$(\Xi^r_{\pi}\otimes Id_A) \Delta(a)=\Delta(a)$, or, equivalently,
$\Delta(a)\in A^r_{\Pi}\otimes A$. Using the elementary remark
that, for subspaces $A$ and $B$ of a vector space $C$, the
intersection of $A \otimes C$ and $C \otimes B$ in $C \otimes C$
is $A \otimes B$ to combine these equivalent statements, one sees
that they occur precisely when $\Delta(a)\in A^r_{\Pi}\otimes
A^l_{\Pi}$.

(g) Let $B_1=\{a\in A\;|\; \Delta(a)\in A^l_{\Pi}\otimes
A^l_{\Pi}\}$, $B_2=\{a\in A\;|\; \Delta(a)\in A^r_{\Pi}\otimes
A^r_{\Pi}\}$ and $B_3=\{a\in A\;|\; \Delta(a)\in A^r_{\Pi}\otimes
A^l_{\Pi}\}$.

Clearly $\widetilde{A_{\Pi}}$ is a subset of $B_i$ for $i=1,2,3$.
By (f), $B_1$ and $B_2$ are subsets of $B_3$. It remains to show
that $B_3\subset \widetilde{A_{\Pi}}$.

Suppose that $a\in B_3.$ Then $\Delta(a)=\sum a_1\otimes a_2$ for
some $a_1\in A^r_{\Pi},a_2\in A^l_{\Pi}$. Consequently,
\begin{eqnarray}a=\sum \epsilon(a_1)a_2=\sum a_1\epsilon(a_2)\in
A^r_{\Pi}\cap A^l_{\Pi}=A_{\Pi}.\label{wow}\end{eqnarray} By (c),
$$\Delta(a)\in (A\otimes A^r_{\Pi})\cap
(A^l_{\Pi}\otimes A)\cap (A^r_{\Pi}\otimes A^l_{\Pi})=
A_{\Pi}\otimes A_{\Pi}.$$ Thus $a\in \widetilde{A_{\Pi}}$.

(h) If $b \in \widetilde{A_{\Pi}}$ then $b \in A_{\Pi}$ by
(\ref{wow}).

Since $A_{\Pi}$ is a subalgebra of $A$, $\widetilde{A_{\Pi}}$ is
also a subalgebra of $A$. Since $S(A_{\Pi})\subseteq A_{\Pi}$,
$S(\widetilde{A_{\Pi}})\subseteq \widetilde{A_{\Pi}}$. To prove
$\widetilde{A_{\Pi}}$ is a Hopf subalgebra it remains to show
that $\widetilde{A_{\Pi}}$ is a subcoalgebra.

Let $a\in \widetilde{A_{\Pi}},$ so that $\Delta(a)=\sum a_1\otimes
a_2 \in A_{\Pi}\otimes A_{\Pi}$. Since $\Delta(A_{\Pi})\subset
A^l_{\Pi}\otimes A^r_{\Pi}$ and $(Id_A\otimes
\Delta)\Delta(a)=(\Delta\otimes Id_A)\Delta(a)$,
$$(Id_A\otimes \Delta)\Delta(a) \in
(A^l_{\Pi}\otimes A^r_{\Pi}\otimes A_{\Pi})\cap
(A_{\Pi}\otimes A^l_{\Pi}\otimes A^r_{\Pi})
=A_{\Pi}\otimes A_{\Pi}\otimes A_{\Pi}.$$
Thus $\Delta(a_1),\Delta(a_2)\in A_{\Pi}\otimes A_{\Pi}$.
Therefore, $a_1,a_2\in \widetilde{A_{\Pi}}$ and
$\widetilde{A_{\Pi}}$ is a coalgebra.

The claims in the second sentence follow from the first part
together with (g).

Suppose finally that $A^l_{\Pi} = A^r_{\Pi}.$ Then by part (c)
$A_{\Pi}=A^l_{\Pi}=A^r_{\Pi}$ is a sub-coalgebra of $A$. Also by
part (e), $S(A_{\Pi})\subseteq A_{\Pi}$. Hence $A_{\Pi}$ is a
Hopf subalgebra of $A$. Thus $\widetilde{A_{\Pi}}=A_{\Pi}$.
\end{proof}

Write $G(A^\circ)$ for the group of all group-like elements in
$A^\circ$. So $G(A^\circ)$ is the character group of $A;$ that is,
by definition, it is $\Hom_{k{\text{-}}\mathrm{alg}}(A,k),$
furnished with the convolution product $*,$ \cite[I.9.24]{BG2}.
We record the following obvious facts:

\begin{lemma}
Suppose $\langle \Pi \rangle$ is a finite abelian subgroup of $G(A^\circ)$.
\begin{enumerate}
\item
$G^r_{\Pi}$ and $G^l_{\Pi}$ are finite abelian groups, isomorphic
to $\langle \Pi \rangle$.
\item
$G_{\Pi}$ is a factor of the direct product of $G^l_{\Pi}$
and $G^r_{\Pi}$, and in particular is abelian of order at
most $|G^l_{\Pi}|^2 = |\langle \Pi \rangle|^2.$
\item
If the order $|\langle \Pi \rangle|$ is coprime to the
characteristic $p$ of $k$ when $p > 0$, then the character
group $\widehat{G^l_{\Pi}}$ is isomorphic to $G^l_{\Pi}$.
\end{enumerate}
\end{lemma}

\begin{proof} (a) It is easy to check that the maps from
$G(A^\circ)$ to $\Aut_{k-\mathrm{alg}}(A)$ given by
$\pi \mapsto \Xi^r_{\pi}$ and $\pi \mapsto \Xi^l_{\pi}$
are, respectively, a group monomorphism and a group
anti-monomorphism. Thus the assertion follows.

(b) This is immediate from Proposition \ref{yysec2.1}(d).

(c) This is standard.
\end{proof}

\subsection{Strongly graded property}
\label{yysec2.2}
Recall our standing hypotheses for $\S$ 2. In addition we assume
in this paragraph and for the remainder of this paper that $S$ is
bijective; since this holds when $A$ is semiprime or satisfies a
polynomial identity \cite{Sk}, it will be the case in the
applications we have in mind. Moreover, we assume in this paragraph
that
\begin{eqnarray}
\langle \Pi \rangle  \textit{ is a finite abelian subgroup of }
G(A^\circ),
\label{yyE2.2.1}
\end{eqnarray}
and that
\begin{eqnarray}
|\langle \Pi \rangle|
\textit{ is coprime to the characteristic } p \textit{ of } k,
\textit{ when } p > 0.
\label{yyE2.2.2}
 \end{eqnarray}

It will be convenient to write  $\widehat{G^l_{\Pi}}$ additively.
For each $\chi\in \widehat{G^l_{\Pi}}$, put
$$A^l_{\Pi,\chi}:=\{a\in A\;|\; g(a)=\chi(g)a, \forall\; g
\in G^l_{\Pi}\}.$$
Then $A$ is a $\widehat{G^l_{\Pi}}$-graded algebra:
$A=\oplus_{\chi\in\widehat{G^l_{\Pi}}} A^l_{\Pi,\chi}$. In
particular $A^l_{\Pi,0}$ consists of the $G^l_{\Pi}$-invariants
of $A$, which we denoted by $A^l_{\Pi}$ in \S \ref{yysec2.1}.
Similarly, $A$ is $G^r_{\Pi}-$graded, with components
$A^r_{\Pi,\chi}$ for $\chi \in \widehat{G^r_{\Pi}}$. The following
is well-known:

\begin{lemma}
Let $G$ be a finite abelian group acting faithfully and semisimply
on an algebra $A$, so $A$ is $\widehat{G}-$graded,
$A=\oplus_{\chi\in \widehat{G}}A_{\chi}$.  If, for any $A_\chi\neq 0$
and $A_{\chi'}\neq 0$, $A_\chi A_{\chi'}\neq 0$ (e.g., $A$ is a
$\widehat{G}$-graded domain), then $A_{\chi} \neq 0$ for all
$\chi \in \widehat{G}$.
\end{lemma}

\begin{proof}
Let $K=\{\chi\in \widehat{G}\;|\; A_{\chi}\neq 0\}$. Since, for any
$A_\chi\neq 0$ and $A_{\chi'}\neq 0$, $A_\chi A_{\chi'}\neq 0$,
$K$ is a finite subsemigroup of $\widehat{G}$, whence $K$ is a
subgroup of $\widehat{G}$. Since $G$ acts faithfully on $A$,
$L:=\{g\in G\;|\; \chi(g)=1, \; \forall \chi\in K\}$ is trivial.
If $K\neq \widehat{G}$, then $L\neq \{1\}$, yielding a contradiction.
\end{proof}

The next result is fundamental to our analysis. We have chosen to
give a direct proof, but see Remark (b) below for an alternative
approach.

\begin{proposition}
Suppose \eqref{yyE2.2.1} and \eqref{yyE2.2.2}.
\begin{enumerate}
\item
$\Delta (A^l_{\Pi,\chi}) \subseteq A^l_{\Pi,\chi} \otimes A$ for
all $\chi\in \widehat{G^l_{\Pi}}$.
\item
$A=\oplus_{\chi\in \widehat{G^l_{\Pi}}}A^l_{\Pi,\chi}$ is
strongly $\widehat{G^l_{\Pi}}$-graded.
\end{enumerate}
\end{proposition}

\begin{proof} (a) The assertion is equivalent to the statement
that $(\Xi^l_{\pi}\otimes Id_A)\Delta (a) =\chi(\Xi^l_{\pi})
\Delta (a)$ for every $a\in A^l_{\chi}$, which follows from
Proposition \ref{yysec2.1}(c).

(b) For this, it suffices to show that
$A^l_{\Pi,\chi} A^l_{\Pi,-\chi}=A^l_{\Pi,0}$ for all
$\chi\in \widehat{G^l_{\Pi}}$. Since $A$ is
$\widehat{G^l_{\Pi}}$-graded, we only need to show that
$A^l_{\Pi,\chi} A$ contains $1$ for all $\chi\in
\widehat{G^l_{\Pi}}$.

Let $\chi\in \widehat{G^l_{\Pi}}$ and suppose there is
$0 \neq a\in A^l_{\Pi,\chi}$. By (a), $\Delta(a)=\sum a_1\otimes a_2$
with $a_1\in A^l_{\Pi,\chi}$. Since $a=\sum \epsilon(a_1)a_2$,
there is a $b\in A^l_{\Pi,\chi}$ such that $\epsilon(b)\neq 0,$
and after adjusting $a$ we may assume that $\epsilon(b)=1$. Then
$\Delta(b)=\sum b_1\otimes b_2$ with $b_1\in A^l_{\Pi,\chi}$.
Therefore $1=\epsilon(b)=\sum b_1S(b_2)\in A^l_{\Pi,\chi}A$.

In the above paragraph, we have shown that $A$ is strongly $K$-graded
where $K=\{\chi\in \widehat{G^l_{\Pi}}\;|\; A_{\Pi,\chi}\neq 0\}$.
As a consequence, if $A_{\Pi,\chi}\neq 0$ and $A_{\Pi,\chi'}\neq 0$,
then $A_{\Pi,\chi}A_{\Pi,\chi'}\neq 0$. By Lemma \ref{yysec2.1}(a),
$G^l_{\Pi}$ acts faithfully on $A$ and by \eqref{yyE2.2.2},
$G^l_{\Pi}$ acts semisimply on $A$. Lemma \ref{yysec2.2} above says
that $K=\widehat{G^l_{\Pi}}$.
\end{proof}

\begin{remarks} (a) Of course the right-hand version of the proposition also
holds.

(b) In fact the proposition can be deduced from standard results on
Hopf-Galois extensions, as follows. Suppose that $\pi : H
\rightarrow \overline{H}$ is an epimorphism of Hopf algebras, with
$\overline{H}$ finite dimensional. So $H$ is a right
$\overline{H}-$comodule algebra via the map $\phi := (1 \otimes \pi)
\circ \Delta : H \rightarrow H \otimes \overline{H}.$ Define $B :=
H^{co \overline{H}} = \{b \in H : \phi (b) = b \otimes 1 \}.$ Then
the inclusion of algebras $B \subseteq H$ is $\overline{H}-$Galois
by \cite[Example 3.6]{Sc} and \cite[Theorem 8.3.1]{Mo}. We apply
this with $A = H$ and $\overline{H} = (kG^l_{\Pi})^*$ Then
$\overline{H}$ is the group algebra $k \widehat{G^l_{\Pi}}$ by
(\ref{yyE2.2.1}) and (\ref{yyE2.2.2}), and so $A$ is strongly
$\widehat{G^l_{\Pi}}-$graded by \cite[Theorem 8.1.7]{Mo}.
\end{remarks}

\subsection{Classical components}
\label{yysec2.3}
At this point we need to recall and refine some definitions
introduced in \cite[\S 7]{LWZ}; at the same time we take the
opportunity to simplify our notation.

\begin{definition}
Let $A$ be an AS-Gorenstein Hopf algebra with bijective antipode,
and let $\pi$ be the algebra homomorphism from $A$ onto
$A/\mathrm{r}.\ann(\int^l_A),$ where $\int^l_A$ denotes the left
integral of Definition \ref{yysec1.3}.
\begin{enumerate}
\item
$A^l_{\pi}$, $A^r_{\pi}$ and $A_{\pi}$ are respectively called
the {\it left classical component}, the {\it right classical
component}, and the {\it classical component} of $A$. These
subalgebras of $A$ will respectively be denoted by $A^l_0$,
$A^r_0$ and $A_0.$
\item
The \emph{integral order} $io(A)$ of $A$ is the order of $\pi$
in the group $G(A^\circ)$.
\item
The \emph{integral annihilator} of $A$ is the maximal Hopf ideal
(see Remark (c) below) contained in $\mathrm{r.ann}(\int^l_A)$. It
is denoted by $J_{iq}$.
\item
The Hopf algebra $A/J_{iq}$ is called the \emph{integral quotient
of} $A$, denoted by $A_{iq}.$
\end{enumerate}
\end{definition}

\begin{remarksA} (a) In \cite[\S
7]{LWZ}, $A^r_{0}$ is called the classical component of $A$. But, as
we'll see in Examples \ref{yysec3.2} and \ref{yysec3.4}, $A^r_{0}$ is
not always a Hopf subalgebra of $A$. We'll provide evidence below in
Theorems \ref{yysec4.1} and \ref{yysec5.2} that $A_{0}$ may always
be a normal Hopf subalgebra of $A$ when $A$ is prime regular affine
of GK-dimension one, so the present definition seems more
appropriate.

(b) By Lemma \ref{yysec2.1}(a), $io(A) = |G^l_{\pi}|=|G^r_{\pi}|$.

(c) It is easy to see that there is a unique largest Hopf ideal
contained in any (two-sided) ideal of a Hopf algebra, so $J_{iq}$ is
well-defined. Since the antipode $S$ is bijective, $S(\int^l_A) =
\int^r_A$, and so the definitions of the integral annihilator (and
thus the integral quotient) are right-left symmetric: $J_{iq}$ is
the biggest Hopf ideal in $\mathrm{l.ann}(\int^r_A)$. In a similar
fashion one can check that the various classical components can all
equally well be defined using the left annihilator of the right
integral.

(d) It will be convenient later (see Hypotheses (\ref{yysec2.5})) to
extend the use of the notation $A^l_0, A^r_0, A_0$ in Definition
\ref{yysec2.3}(a) to cases where the algebra homomorphism $\pi$ is
not necessarily given by the left integral.
\end{remarksA}

Let us now review some basic facts about how the above concepts
behave in the case of a Hopf algebra which is a finite module over
its center. Recall that, if $Z$ is an Ore domain, then the
\emph{rank} of a $Z$-module $M$ is defined to be the
$Q(Z)-$dimension of $Q(Z)\otimes_Z M$, where $Q(Z)$ is the quotient
division ring of $Z$; the rank of a right $Z$-module is defined
analogously. Let $R$ be any algebra. The PI-degree of $R$ is defined
to be
$$\PIdeg(R)=\min\{n\;|\; R\hookrightarrow M_n(C)
\textrm{ for some commutative ring } C\}$$
(see \cite[Ch. 13]{MR}). If $R$ is a prime PI ring with center $Z$,
then the PI-degree of $R$ equals the square root of the rank of $R$
over $Z$.

Part (b) below was previously known only under the extra hypothesis
that $A$ is regular \cite[Lemma 5.3(g)]{LWZ}.

\begin{theorem} Let $A$ be an affine Hopf $k$-algebra which is a
finite module over its center.
\begin{enumerate}
\item
$A$ is noetherian, AS-Gorenstein and has a bijective antipode.
\item
The integral order $io(A):= n$ of $A$ is finite.
\item If $A$ is prime and $Z(A)$ is integrally closed, then $io(A)
\leq \PIdeg(A).$ In particular this holds if $A$ is regular
and prime.
\item
The integral quotient $A_{iq}$ is isomorphic as a Hopf algebra
to $(kC_n)^{\circ}$, the Hopf dual of the group algebra of the
cyclic group of order $n$.
\item
$J_{iq} = \cap_{i=0}^{n-1} \mathrm{ker}((\pi)^i: A \rightarrow k)$.
\item Suppose that $A$ is regular of global dimension $d$, and that
$n$ is a unit in $k.$
Then $A_0^l$ and $A_0^r$ are both regular of global dimension $d.$
\end{enumerate}
\end{theorem}

\begin{proof} (a) The Artin-Tate Lemma \cite[13.9.10]{MR} shows
that $A$ is noetherian. That $A$ is AS-Gorenstein is a special
case of the main result of \cite{WZ1}. Bijectivity of the antipode
is proved in \cite{Sk}.

(b) We'll prove that $io(A)$ is bounded above by the number of
maximal ideals of $A$ whose intersection with $Z(A)$ equals $\ker
\epsilon \cap Z(A),$ which of course is finite because $A$ is a
finite $Z(A)-$module.

 By the definition of $\int^l_A$, $\ker
\epsilon \cap Z(A) \subseteq \mathrm{r.ann}(\int^l_A)$. Hence, by
Muller's theorem \cite[Theorem III.9.2]{BG2}, $\ker \epsilon$ and
$\mathrm{r.ann}(\int^l_A)$ belong to the same clique of
$\mathrm{maxspec} A$ (in the sense of \cite[I.16.2]{BG2}, for
example). This means that there is a positive integer $t$ and a set
$k= V_0, \ldots , V_t = \int^l_A$ of simple right $A-$modules such
that, for each $j = 0, \ldots , t-1$, either $\Ext^1_A(V_{j+1},V_j)
\neq 0$ or $\Ext^1_A(V_{j},V_{j+1}) \neq 0$. Now, since $\int^l_A$
is a one-dimensional (right) $A-$module, it has an inverse in
$G(A^{\circ})$, given in fact by $(S(\int^l_A))^{\ast}.$ Hence, the
functor $\int^l_A \otimes -$ is an auto-equivalence of the category
of right $A-$modules. In particular
$\Ext^i_A(M,N)\cong\Ext^i_A(\int^l_A \otimes M, \int^l_A \otimes N)$
for all non-negative integers $i$ and all right $A-$modules $M$ and
$N.$ Applying this to the above chain of $\Ext$-groups, we deduce
that (the annihilators of) $\int^l_A \otimes k \cong \int^l_A$ and
$(\int^l_A)^{\otimes 2}$ are in the same clique. That is,
$(\int^l_A)^{\otimes 2}$ is in the same clique as $k$. Repeating
this argument allows us to deduce that $(\int^l_A)^{\otimes m}$ is
in the clique of $k$ for all non-negative integers $m$. Now the
finiteness of clique($\ker \epsilon $) noted at the start of the
argument forces $(\int^l_A)^{\otimes r}\cong(\int^l_A)^{\otimes s}$
for non-negative integers $r,s$ with $r > s$. Thus
$(\int^l_A)^{\otimes (r-s)}\cong k$ with $r-s \leq
|\mathrm{clique}(\ker \epsilon )|,$ as required.

(c) The extra hypotheses ensure that $A$ equals its trace ring, so
that
$$|\mathrm{clique}(\ker \epsilon )| \leq \PIdeg(A)$$
by \cite[Theorem 8]{Bra}. When $A$ is regular
its center is integrally closed by \cite[Lemma 5.3(b)]{LWZ}.

(d),(e) These are special cases of \cite[Lemma 4.4(b),(c)]{LWZ}.

(f) By the strong grading property, Proposition \ref{yysec2.2}(b),
$A^l_0$ is a direct summand of $A$ as left $A^l_0-$module. So, if
$V$ is any left $A^l_0-$module, then $$\prdim_{A^l_0}(V) \leq
\prdim_{A^l_0}(A \otimes_{A^l_0} V) \leq \prdim_{A}(A
\otimes_{A^l_0} V),$$ where the second inequality holds because $A$
is a projective left $A^l_0 -$module. Therefore $\gldim(A^l_0) \leq
d.$ However, since $A$ is a finitely generated $A^l_0-$module, these
algebras have the same GK-dimension, and this is $d$ thanks to the
Cohen-Macaulay property satisfied by $A$, \cite[Theorem A]{BG1}.
Since Krull dimension equals GK-dimension and is a lower bound for
global dimension for noetherian rings which are finite over their
centers \cite{GS}, the result follows.
\end{proof}

\begin{proposition}
Let $A$ be an AS-Gorenstein Hopf algebra with a bijective antipode.
Assume \eqref{yyE2.2.2} for $\Pi = \{ \pi \}$, where $\pi : A
\rightarrow A/\mathrm{r.ann}(\int^l_A).$ Then
$$J_{iq} = (J_{iq}\cap A^l_0)A =  A(A^l_0 \cap J_{iq})
=(\ker{\epsilon} \cap A^l_0)A.$$
\end{proposition}

\begin{proof}
By Theorem \ref{yysec2.3}(e), $J_{iq}$ is stable under
the action of $G^l_{\pi}$ on $A$. Thus $J_{iq}$ is
$\widehat{G^l_{\pi}}$-graded. The first two equalities therefore
follow at once from strong grading, Proposition \ref{yysec2.2}(b).
Moreover, since $J_{iq} \subseteq \ker \epsilon$,  $J_{iq} \subseteq
(\ker{\epsilon} \cap A^l_0)A.$ But the factors by both these right
ideals have $k$-dimension $n$, by Proposition \ref{yysec2.2}(b) and
Theorem \ref{yysec2.3}(d), so equality follows.
\end{proof}

\begin{remarksB} (a) The hypothesis
\eqref{yyE2.2.2} that $p\nmid io(A)$ is necessary for the validity
of Proposition \ref{yysec2.2}. For, let $k$ have characteristic $p
> 0$, and let $A$ be the factor of the enveloping algebra of the
$k-$Lie algebra $\mathfrak{g} = \{kx \oplus ky : [y,x] = x \}$ by
the Hopf ideal $\langle y^p - y \rangle.$ 
Continue to write $x$ and $y$ for the images of these elements in
$A$. Then, by \cite[Lemma 6.3(c)]{BZ}, $\Xi_{\pi}(x)=x$ and
$\Xi_{\pi}(y)=y-1,$ so that $G^l_{\pi} = G^r_{\pi} =\langle
\Xi_{\pi} \rangle \cong C_p,$ and $A_0^r = A^l_0 = A_0 = k[x]$. Thus
$A = \bigoplus_{i=0}^{p-1} A_0 y^i,$ but this is not a strong group
grading. In fact $A$ has no non-trivial ${\mathbb Z}_p$-grading that
is a strong grading. Observe also that this example is regular (and
so hereditary): for by \cite{LL} it's enough to check that the
trivial $A-$module has finite projective dimension, and this follows
since it is a direct summand of the semisimple module $A/Ax,$ which
manifestly has projective dimension one. Notice finally that $J_{iq}
= \langle x \rangle,$ so that Proposition \ref{yysec2.3} remains
valid for this example. We do not know if (\ref{yyE2.2.2}) is
actually needed for Proposition \ref{yysec2.3}, or for Theorem
\ref{yysec2.3}(f).

(b) Given the known examples, we ask:

Question: Under the hypotheses of Proposition \ref{yysec2.3}, is the
center $Z(A) $ of $A$ contained in the classical component $A_0$?

We shall show in Lemma
\ref{yysec5.2} that this is the case when $A$ is prime regular of
GK-dimension 1, but it may be that it is true much more generally.

(c) Of course the version of Proposition \ref{yysec2.3} with $A_0^r$
replacing $A^l_0$ is also true.

(d) In fact (d) and (e) of Theorem \ref{yysec2.3} hold, with the
proofs from \cite{LWZ}, in any Hopf algebra with finite integral
order.

(e) We do not know if all the parts of Theorem \ref{yysec2.3} remain
valid if the hypotheses are weakened to include all affine
noetherian prime Hopf $k$-algebras which satisfy a polynomial
identity. We note here that the question is not vacuous, since not
all such algebras are finite modules over their centers. Thus the
algebra $\bar{U}$ constructed in \cite{GL} as the bosonisation of
the enveloping algebra of a Lie superalgebra is such an algebra, of
GK-dimension 2, which is \emph{not} a finite module over its center
(as can be most quickly confirmed by checking that the clique of the
augmentation ideal is infinite).
\end{remarksB}

\subsection{The integral minor}
\label{yysec2.4}
In addition to the integral order, our classification of Hopf
algebras of GK-dimension one will involve a second invariant,
which - crudely speaking - is a measure of the extent of
left-right asymmetry between the right and left hit actions
of $\pi$, the group-like element arising from the left
integral. More precisely, it measures the difference between
$G_{\pi}$ and $G^l_{\pi} \times G^r_{\pi}.$  The definition is
as follows.

\begin{definition}
Let $A$ be a Hopf algebra and keep the notation and hypotheses
of Definition \ref{yysec2.3}. Recall from \ref{yysec2.1} that
$G_{\pi}$ is an abelian subgroup of
$\mathrm{Aut}_{k-\mathrm{alg}}(A)$ generated by its subgroups
$G^l_{\pi}$ and $G^r_{\pi}.$ Suppose that the integral order
$io(A)$ is finite; or, equivalently, that $G_{\pi}$ is a finite
group. The {\it integral minor} of $A$ is
$$im(A):=|G^l_{\pi}/(G^l_{\pi}\cap G^r_{\pi})|,$$
the intersection being taken as subgroups of $G_{\pi}.$
\end{definition}

In particular, $im(A)=1$ if and only if $G^l_{\pi}=G^r_{\pi},$
and this holds in particular when $A$ is
cocommutative.

\subsection{Classical components in GK-dimension one}
\label{yysec2.5}
In preparation for the work on Hopf algebras of  GK-dimension
one which will occupy the rest of the paper, we state here as
our starting point what can immediately be written down about
such an algebra using the results in \cite{LWZ} and the concepts
introduced in this section.

\begin{theorem}
Let $A$ be a semiprime affine Hopf $k-$algebra of GK-dimension
one, with $\pi$ the homomorphism from $A$ onto
$A/\mathrm{r.ann}(\int^l_A)$.
\begin{enumerate}
\item
The integral order $io(A):=n$ of $A$ is finite.
\item
If \eqref{yyE2.2.2} holds for $\Pi=\{\pi\}$, then $A$ is
strongly graded in two ways by the cyclic group of order $n$,
namely
$$ A=\oplus_{\chi\in \widehat{G^l_{\pi}}}A^l_{\pi,\chi} \textit{ and }
A=\oplus_{\chi\in \widehat{G^r_{\pi}}}A^r_{\pi,\chi}.$$
\item
The injective dimension of $A$ is one.

\noindent Suppose in addition for (d), (e) and (f) that $A$ is
regular; for (e) and (f), assume also that it is prime, and
for (f) assume also \eqref{yyE2.2.2}.
\item
$A$ is hereditary.
\item
$\PIdeg(A) = n.$
\item
$A^l_0,$ $A^r_0$ and $A_0$ are affine commutative Dedekind domains,
with $A_0^l \cong A_0^r.$
\end{enumerate}
\end{theorem}

\begin{proof} By \cite{SSW}, $A$ is noetherian and is a finite
module over its center. Thus $A$ is AS-Gorenstein and Cohen-Macaulay,
by \cite[Theorems 01. and 0.2(1)]{WZ1}, so $A$ has a left integral
and $\pi$ is defined for $A$.

(a) follows from Theorem \ref{yysec2.3}(b). (b) is immediate from
Proposition \ref{yysec2.2}. (c) follows from \cite[Theorem
0.2(3)]{WZ1}.

Suppose henceforth that $A$ is regular. Then (c) immediately gives
(d).

(e) That the PI-degree of $A$ equals its integral order when $A$ is
prime and regular is part of \cite[Theorem 7.1]{LWZ}.

(f) That $A^l_0$ is an affine commutative domain is proved in
\cite[Theorem 7.1]{LWZ}; that it is Dedekind now follows from
Theorem \ref{yysec2.3}(f). An isomorphism of the right and left
classical components is provided by the antipode, thanks to
Proposition \ref{yysec2.1}(e). Finally, $A_0 = (A^l_0)^{G^r_{\pi}} =
(A^r_0)^{G^l_{\pi}}$ by Proposition \ref{yysec2.1}(d). Thus $A_0$ is
an affine commutative domain, and since the fixed ring of an
integrally closed commutative domain under the action of a finite
group is again integrally closed, it follows that $A_0$ is Dedekind.
\end{proof}

\begin{remark} The equality in (e) fails when any one of the hypotheses
that $A$ is prime, regular, or has GK-dimension one, is dropped. To
see the need for the prime hypothesis, one can take $A = \C G$ where
$G$ is any direct product of a finite group with an infinite cyclic
group. In this case $\int^l_A \cong k$ and so $io(A) = 1$. The
regular hypothesis fails in Case 2 of Example \ref{yysec3.2} in the
next section; here, the integral order is 1 but the PI-degree is 2.
Finally, (e) fails for prime regular affine PI Hopf algebras of
GK-dimension $\geq 2$, see \cite[Example 8.5]{LWZ}. Another example
is the enveloping algebra of $\mathfrak{sl}(2,k)$ over a field $k$
of positive characteristic $p$; in this case $H:=
U(\mathfrak{sl}(2,k))$ is unimodular, that is $\int^l_H \cong k$,
\cite[Proposition 6.3(e)]{BZ}. However it is possible that the
equality (e) holds for all regular affine PI Hopf domains of
GK-dimension 2.
\end{remark}

At this point we can list all the hypotheses that we will use for
our classification. By the above theorem, given (a), (d) and (e)
in the list below, the map $\pi: H\to H/\mathrm{r.ann}(\int^l_H)$
satisfies (b) and (c).

\begin{hypotheses}
\begin{enumerate}
\item $H$ is an affine, noetherian, prime, PI Hopf algebra; \item
there is an algebra homomorphism $\pi: H\to k$ such that the order
$n$ of $\pi$ in $G(H^{\circ})$ is equal to the PI-degree of $H$;
\item the
invariant subring $H^l_{0}$ under the action of the left winding
automorphism $\Xi^l_{\pi}$ is a Dedekind domain; \item when
$p=\text{char}\; k>0$, then $n$ is coprime to $p$; \item $\GKdim
H=\gldim H=1$.
\end{enumerate}
\end{hypotheses}

We remark for future use that the analogues of Theorems
\ref{yysec2.3} and
\ref{yysec2.5} remain valid (with the same proofs) for any map $\pi$
satisfying the above hypotheses.

\section{Examples}
\label{yysec3} In this section we assemble some examples of prime
affine Hopf $k$-algebras of GK-dimension one. Note that we've
already listed one such example in Remarks \ref{yysec2.3}B(a).

\subsection{Commutative examples}
\label{yysec3.1} Recall \cite[Theorem 20.5]{Hu} that, over an
algebraically closed field $k$,
there are precisely two connected algebraic groups of dimension one.
Therefore there are precisely two commutative $k$-affine domains of
GK-dimension one which admit a structure of Hopf algebra, namely
$A_1 = k[X]$ and $A_2 = k[X^{\pm 1}]$. For $A_1,$ $\epsilon (X) =
0,$ $S(X) = -X,$ and
$$\Delta (X) = X \otimes 1 + 1 \otimes X. $$
For $A_2,$ $\epsilon (X) = 1,$ $S(X) = X^{-1},$ and
$$ \Delta (X) = X \otimes X. $$
Commutativity implies that $ io(A_j) = im(A_j) = 1$ for $j = 1,2.$
Of course $A_1$ and $A_2$ are regular, in common with \emph{all}
commutative affine Hopf algebras which are domains, \cite[Theorem
11.6]{Wa}. The following result is \cite[Theorem 0.2(c)]{LWZ}, which
takes care of the case when $io(H)=1$.

\begin{proposition}
Suppose $H$ satisfies Hypotheses \ref{yysec2.5}. If $\pi: H\to
H/\mathrm{r.ann}(\int^l_H)$ has order 1 (or equivalently,
$\int^l_H\cong k$ as $H$-bimodule), then $H$ is isomorphic as a Hopf
algebra to $A_1$ or $A_2$.
\end{proposition}

\subsection{Dihedral group algebra}
\label{yysec3.2}
Let ${\mathbb D}$ denote the infinite dihedral group
$\langle g,x | g^2=1, gxg=x^{-1}\rangle$. The group algebra
$k{\mathbb D}$ is an affine cocommutative Hopf
algebra, which is prime by \cite[Theorem 4.2.10]{Pa}. Being
a skew group algebra of a finite group over the
Laurent polynomial algebra $k[x^{\pm 1}],$ $k{\mathbb D}$ has
GK-dimension one, \cite[Proposition 8.2.9(iv)]{MR}.
By \cite[Theorem 10.3.13]{Pa}, $k{\mathbb D}$ is regular if
and only if the characteristic of $k$ is different
from 2. By cocommutativity,
$$ im(k{\mathbb D}) = 1.$$

Since it is a free module of rank 2 over $k[x^{\pm 1}]$,
$k{\mathbb D} \subseteq M_2(k[x^{\pm 1}])$, so that
$\PIdeg(k{\mathbb D}) \leq 2$. As the algebra
is not commutative,
$$\PIdeg(k{\mathbb D}) = 2.$$
Either by direct calculation, or using \cite[Lemma 6.6]{BZ}, one
sees that, as a \emph{right} module,
\begin{eqnarray}
\int^l_{k\mathbb{D}} \cong k\mathbb{D}/\langle x-1, g+1 \rangle.
\label{yyE3.2.1}
\end{eqnarray}

\noindent
\textbf{Case 1:}
Suppose that the characteristic of $k$ is not 2. It follows from
\eqref{yyE3.2.1} that $G^l_{\pi}= G^r_{\pi} = \{1,\tau \} \cong C_2$,
where $\tau(x) = x$ and $\tau(g) = -g.$ Thus the classical components
(necessarily equal on account of cocommutativity), are the elements
fixed by $\tau,$ namely
$$ (k\mathbb{D})_0 =(k\mathbb{D})^l_0 =(k\mathbb{D})^r_0=
k[x^{\pm 1}]. $$ Finally, $io(k\mathbb{D})=2,$ and $J_{iq} =
(x-1)k\mathbb{D}$, so that $(k\mathbb{D})_{iq} = k\langle g
\rangle$, the group algebra of the cyclic group of order 2.

\medskip
\noindent \textbf{Case 2:} Now let $k$ have characteristic 2. Then
\eqref{yyE3.2.1} shows that $k\mathbb{D}$ is unimodular, so
$G^r_{\pi} = G^l_{\pi} = \{1\}$. Thus $io(k\mathbb{D})=1,$ and the
classical components all equal $k\mathbb{D},$ with $J_{iq}$ equal to
the augmentation ideal and $(k\mathbb{D})_{iq} = k.$

\subsection{Taft algebras}
\label{yysec3.3} \cite[Examples 2.7, 7.3]{LWZ} Let $n$ and $t$ be
integers with $n$ greater than 1, and coprime to the characteristic
of $k$ if the latter is positive, and with $0 \leq t \leq n-1.$ Fix
a primitive root $n$th root $\xi$ of $1$ in $k$. Let $H:=H(n,t,\xi)$
be the $k$-algebra generated by $x$ and $g$ subject to the relations
$$g^n=1, \quad\text{and}\quad xg=\xi gx.$$
Then $H$ is a Hopf algebra with coalgebra structure given by
$$\Delta(g)=g\otimes g, \; \epsilon(g)=1,
\quad\text{and}\quad \Delta(x)=x\otimes g^t+1\otimes x, \;
\epsilon(x)=0,$$ and with $$S(g) = g^{-1} \textit{ and } S(x) = -x
g^{-t}.$$ Thus one can think of $H(n,t,\xi)$ either as a (slightly
generalised) factor of the quantum Borel of $\mathfrak{sl}(2,k)$ at
an $n$th root of unity \cite[I.3.1]{BG2}, or as a limit of a family
of finite dimensional Taft algebras, \cite[1.5.6]{Mo}.

Viewing $H(n,t,\xi)$ as the skew group algebra $k[x] \ast \langle g
\rangle$ lets us see easily that it is affine noetherian, prime,
hereditary, and has GK-dimension one \cite{LWZ}. An easy calculation
shows that $Z(H) = k[x^n],$ so that
$$ \PIdeg(H(n,t,\xi)) = n.$$ We leave to the reader the
routine proof of the

\begin{proposition}
Let $n,n',t,t'$ be integers with $n,n'$ greater than 1, and coprime
to the characteristic of $k$ if the latter is positive, and with $0
\leq t,t' \leq n-1.$ Let $\xi$ and $\eta$ be two primitive $n$th
roots of 1 and $\xi'$ a primitive $n'$th root of 1.
\begin{enumerate}
\item
As algebras, $H(n,t,\xi) \cong H(n',t',\xi')$ if and only if $n = n'.$
\item
As bialgebras, $H(n,t,\xi) \cong H(n',t',\xi)$ if and only if $n = n'$
and $t = t'.$
\item
If $\xi=\eta^v$ for some $1\leq v\leq n-1$,
then $H(n,t,\xi)\cong H(n,vt,\eta)$ as Hopf algebras.
\end{enumerate}
\end{proposition}

Recall that $H=H(n,t,\xi)$. Using \cite[Lemma 6.6]{BZ} we find that
$$ \int^l_H \cong H/\langle
x, g-\xi^{-1} \rangle. $$The corresponding homomorphism $\pi$ yields
left and right winding automorphisms \begin{eqnarray*}
 \Xi^r_{\pi}: \begin{cases}
 x \mapsto \xi^{-t}x \\ g \mapsto \xi^{-1}g
 \end{cases} \textit{ and } \qquad
\Xi^l_{\pi}: \begin{cases}
 x \mapsto x \\ g \mapsto \xi^{-1}g,
 \end{cases}
\end{eqnarray*}
so that $G^l_{\pi} = \langle \Xi^l_{\pi} \rangle$ and $G^r_{\pi} =
\langle \Xi^r_{\pi} \rangle$ have order $n$. Hence, as predicted by
Theorem \ref{yysec2.5}(e), $$io(H(n,t,\xi)) = n.$$ The integral
annihilator $J_{iq}$ is $\cap\{\langle x, g-\xi^i \rangle  : 0 \leq
i \leq n-1 \},$ which equals $xH.$ Thus $H_{iq} = H/J_{iq} \cong k
\langle g \rangle.$ As for the right and left classical components,
$$H^r_0 = H^{G^r_{\pi}} = k[xg^{-t}], \textit{ and } H^l_0 =
H^{G^l_{\pi}} = k[x],$$ so one checks that, as predicted by
Proposition \ref{yysec2.1}(c), $H^r_0$ is a left coideal subalgebra
of $H$ and $H^l_0$ is a right coideal subalgebra.

Let $m:=n/\gcd(n,t)$. Then $G^l_{\pi} \cap G^r_{\pi} = \langle
(\Xi^l_{\pi})^m \rangle ,$ so that the integral minor of
$H(n,t,\xi)$ is
\begin{eqnarray}
 im(H(n,t,\xi)) = m.
\label{yyE3.3.1}
\end{eqnarray}
This shows that $G_{\pi} \cong C_n \times C_m$. Now the classical
component
$$H_0 = H^{G_{\pi}} = H^l_0 \cap H^r_0 = k[x^m],$$
and since this is a Hopf subalgebra of $H$ we find that
$\widetilde{H_0}$, the largest Hopf subalgebra of the classical
component, is $H_0=k[x^m]$. Note that
$$Z(H) = k[x^n] \subseteq k[x^m] = \widetilde{H_0},$$
so that $\widetilde{H_0}$ is central if and only if the integral
minor is $n$ - that is, if and only if $\gcd (n,t)= 1$.
Nevertheless, $\widetilde{H_0}$ can easily be checked to be
\emph{normal} in $H$, and to be a maximal commutative Hopf
subalgebra of $H.$

\subsection{Generalised Liu algebras}
\label{yysec3.4} We start with a presentation of generalised Liu
algebras that is convenient for the proof of some algebraic
properties and then convert it into another form in which the
generators of the algebra are more rigid. Both presentations are
helpful in understanding these algebras.

Let $n$ be an integer greater than 1, with $n$ coprime to the
characteristic $p$ of $k$ if $p > 0$. Fix a primitive $n$th root
$\theta$ of one in $k$. Let $w$ be a positive integer with $\gcd
(n,w) := b \geq 1.$ Write $ n = n'b,$ $w = w'b.$ Define a
$k-$algebra $B$ with generators $h^{\pm 1}, f$ and $y,$ and
relations
$$ hh^{-1} = h^{-1}h = 1, \qquad f^b = 1, \qquad hf = fh;$$
$$\begin{aligned}
yh &= \theta hy;\\
yf &= \theta^{n'}fy; \\
y^n - 1 + h^{nw'} &= 0.
\end{aligned}$$

We shall show that $B$ admits a structure of Hopf algebra, but first
we develop some of its properties as a $k-$algebra. Start with the
group algebra
$$A := k\langle h^{\pm 1},\,f \rangle \cong k(C_{\infty} \times
C_b).$$ Let $\sigma$ be the $k-$algebra automorphism of $A$ defined
by $\sigma (h) = \theta h,$ $\sigma (f) = \theta^{n'} f,$ and form the
skew polynomial algebra
$$ C := A[y; \sigma].$$
Since $|\sigma | = n$ and $A$ is commutative, $C$ has PI-degree at
most $n.$ One easily confirms that $A^{\sigma}$
is the group algebra $k\langle x \rangle$, where
\begin{eqnarray}
x := h^{n'}f^{-1},\label{yyE3.4.1}
\end{eqnarray}
so that $$Z(C) = k[x^{\pm 1}, y^n].$$
Thus $y^n - 1 + x^w \in Z(C),$ so that we can form $$B := C/\langle
y^n - 1 + x^w  \rangle.$$ Abusing notation, we continue to write
$h,f,y,x$ for the images of these elements in $B$. Thus, in $B,$
\begin{eqnarray}
y^n = 1- x^w = 1 - h^{n'w} = 1 - h^{nw'}.\label{yyE3.4.2}
\end{eqnarray}
So $B$ has $k-$basis
\begin{eqnarray}
\{h^if^jy^l : i \in \mathbb{Z}, 0 \leq j
\leq b-1, 0 \leq l \leq n-1 \}.\label{yyE3.4.3}
\end{eqnarray} This shows
that \begin{eqnarray} B \textit{ is a finitely generated free (left and
right) }
k\langle h^{\pm 1}\rangle \textit{-module}, \label{yyE3.4.4}
\end{eqnarray} which permits us to confirm firstly that
\begin{eqnarray}
k \langle h^{\pm 1} \rangle \setminus \{0 \} \textit{
consists of regular elements of } B ,
\label{yyE3.4.5}
\end{eqnarray}
and hence, in view of \eqref{yyE3.4.2},
\begin{eqnarray} y \textit{ is not a zero divisor in } B.
\label{yyE3.4.6}
\end{eqnarray} Finally, it follows from \eqref{yyE3.4.2}
and \eqref{yyE3.4.3} that
\begin{eqnarray}
Z(B) = k[x^{\pm 1}].
\label{yyE3.4.7}
\end{eqnarray}
It's clear from its construction that $B$ is a noetherian algebra
with $\GKdim(B) = 1.$

We turn now to the coalgebra structure of $B.$ For the counit,
define
$$\epsilon(f)= \epsilon(h)=1,\quad \epsilon(y)=0;$$
trivially, this is an algebra homomorphism. For the coproduct, we
propose to define$$
\begin{aligned} \Delta(h)&=h\otimes h,\quad
\Delta(f)=f\otimes f,\\
\Delta(y)&=y\otimes g + 1\otimes y,\end{aligned}$$ for a suitable
choice of  a group-like element $g \in A = \langle h,f \rangle.$
Since $g$ is group-like, $g=h^a f^b$ for some integers $a,b$. Now we
require
$$\Delta (y)^n = \Delta (y^n) = \Delta(1 - x^w) = (1- x^w) \otimes x^w + 1
\otimes (1-x^w).$$
Thus, by the $q-$binomial theorem
\cite[I.6.1]{BG2}, we require
\begin{eqnarray} \quad \quad g^n = x^w = h^{w'n} \quad
\textit{ and } \quad \sigma (g) = \xi g
{\text{ for a primitive $n$th root $\xi$ of 1.}}
\label{yyE3.4.8}
\end{eqnarray}
The first of these constraints implies that $g = h^{w'}f^i$ for some
$i$, $0 \leq i \leq b-1.$ Then $\sigma (g) = \theta^{w' + n'i}g$ so
the second part of \eqref{yyE3.4.8} requires that
\begin{eqnarray}
\gcd (n, w' + n'i) = 1. \label{yyE3.4.9}
\end{eqnarray}
By the elementary lemma at the end of this subsection, such a
solution $i = i_0$ to \eqref{yyE3.4.9} exists, and accordingly we
set
\begin{eqnarray}
g := h^{w'}f^{i_0}\label{yyE3.4.10}
\end{eqnarray}
in the definition of $\Delta (y).$ With these definitions it is
routine to check that $\Delta$ is an algebra homomorphism, that
it satisfies the coassociativity axioms, and that $\epsilon$ is a
counit. Finally, we define
$$ S(h)=h^{-1},\quad S(f)=f^{-1},\quad S(y)=-yg^{-1}
$$
and check easily that this extends to an anti-automorphism of $B$
satisfying the antipode axiom.

We show next that the algebra $B$ is also generated by $x^{\pm
1},g,y$. Since $\gcd(w'+n'i_0,n)=1$, there are integers $u$ and $v$
such that $u(w'+n'i_0)+v n=1$. Then we calculate that
\begin{eqnarray}
\quad \quad \quad g^{u+vn}x^{(u i_0 - vb(w' - 1))}=
(h^{w'} f^{i_0})^{u+vn}(h^{n'}f^{-1})^{(u i_0 - vb(w' - 1))}
=h.\label{yyE3.4.11}
\end{eqnarray}
Hence $h$ is generated by $g$ and $x^{\pm
1}$, noting that $g^{-1} = g^{n-1}x^{-w}$. Since $f=x^{-1}h^{n'}$,
$B$ is generated by $y,g$ and $x^{\pm 1}$. Combining the above
assertions, $B$ is generated by $x^{\pm 1},g,y$, subject to the
relations
\begin{eqnarray}
\begin{cases}
yg = \xi gy, & \quad \quad y^n = 1-x^w = 1 -
g^n \\
xy=yx, & \quad \quad xg=gx \\ x x^{-1} = 1
\end{cases}
\label{yyE3.4.12}
\end{eqnarray}
where $\xi=\theta^{w' + n' i_0}$ is a primitive $n$th root of 1.
The coalgebra structure of $B$ is determined by the fact that $x$
and $g$ are group-like and $y$ is $(g,1)$-primitive. The Hopf
algebra structure of $B$ is uniquely determined by these data and
therefore $B$ is denoted by $B(n,w,\xi)$.

Note that two different solutions $i_0$ to \eqref{yyE3.4.9} give
distinct $\xi:=\theta^{w'+n'i_0}$ and different formulas of $g$ in
terms of $h$ and $f$ \eqref{yyE3.4.10}. And, replacing $h$ by $hf^j$
for some integer $j$ satisfying $\gcd(1+jn',n)=1$ changes one $i_0$
to another. This means that two different choices of the pair 
$(\theta,i_0)$ may yield the same Hopf algebra $B$.

One advantage of the second presentation of $B$ (see
\eqref{yyE3.4.12}) is that the ordered set $\{n,w,\xi,x,g\}$ can be
recovered uniquely by the Hopf algebra structure of $B$ (see (j) of
the next theorem). In this first presentation, however, $h$ (and
$\theta$) can not be determined uniquely by the Hopf structure of
$B$.

We are now ready to list the main properties of the algebras $B =
B(n,w,\xi)$:

\begin{theorem}
Let the $k-$algebra $B = B(n,w,\xi)$ be defined as above.
Assume that the characteristic of $k$ does not divide $n$.
\begin{enumerate}
\item
As a $k-$algebra, $B$ is generated by $y,h^{\pm 1},f,$ with relations
as stated at the beginning of this section; and it is also generated
by $y,g$ and $x^{\pm 1}$, with relations in \eqref{yyE3.4.12}.
\item
$B$ is an affine noetherian Hopf
algebra of Gel'fand-Kirillov dimension $1,$ with center
$k\langle x^{\pm 1}\rangle.$
\item
$B$ is prime.
\item
$\PIdeg( B) = n.$
\item
$B$
has finite global dimension if and only if $w$ is a unit in
$k$. In this case, $\gldim B = 1$.
\item
$io(B)= im(B)=n.$
\item
The fixed subring $B^{\Xi_{\pi}^l}$ is $k\langle y, x^{\pm 1}
\rangle;$ and $B^{\Xi_{\pi}^r}= k\langle yg^{-1}, x^{\pm 1}
\rangle.$ Thus $$B_0 := B^{\Xi_{\pi}^l} \cap B^{\Xi_{\pi}^r} =
k[x^{\pm 1}] = Z(B).$$
\item
The group of group-like elements of $B$ is $\langle
h,f \rangle,$ isomorphic to $C_{\infty} \times C_b.$
\item
If $z\in B$ is a nonzero $(a,1)$-primitive element for some
group-like element $a\neq 1$, then $a=g$ and $z\in ky+k(1-g)$.
\item
Let $B(\bar{n},\bar{w},\bar{\xi})$ be another Hopf algebra
of the same kind with generators $\bar{x}^{\pm 1},\bar{g}$,
and $\bar{y}$.
If $\phi: B(\bar{n},\bar{w},\bar{\xi})\to B(n,w,\xi)$
is a Hopf algebra isomorphism, then $n=\bar{n},w= \bar{w},
\xi=\bar{\xi}$ and $\phi(\bar{x})=x$,
$\phi(\bar{g})=g$, $\phi(\bar{y})=\eta y$ for some
$n$th root $\eta$ of $1$.
\end{enumerate}
\end{theorem}

\begin{proof} (a) and (b) are clear from the discussion above.

(c) In view of \eqref{yyE3.4.5}, to show that $B$ is prime it suffices to
prove that, if $I$ is a non-zero ideal of $B$, then
\begin{eqnarray}
I \cap k \langle h^{\pm 1} \rangle \neq 0.
\label{yyE3.4.13}
\end{eqnarray}
So let $\alpha$ be a non-zero element of $I$. Using \eqref{yyE3.4.3} we
can uniquely write $$ \alpha = \sum_{l=0}^{n-1}\alpha_ly^l, $$ with
$\alpha_l \in k\langle h^{\pm 1},f \rangle$ and choose $\alpha$ so
that the number of non-zero $\alpha_l$ occurring is as small as
possible for non-zero elements of $I$. We assume for a contradiction
that this number is greater than one. Since $y^n = 1 -
h^{nw'},$ we can multiply $\alpha$ on the right by a suitable
power of $y$ to ensure that $\alpha_0 \neq 0.$ Now $\alpha h - h
\alpha \in I$ is non-zero and has strictly smaller $y$-length than
$\alpha$, contradicting our choice of $\alpha$. Therefore $$ I \cap
k \langle h^{\pm 1}, f \rangle \neq 0.$$

Now $\theta^{n'}$ is a primitive $b$th root of 1 in $k$, so the $b$
primitive idempotents $$ e_j :=
\frac{1}{b}\sum_{i=0}^{b-1}\theta^{n'ji}f^i $$ form a $k$-basis of
$k\langle f \rangle.$ Fix $$0 \neq \beta = \sum_{j=0}^{b-1} \gamma_j
e_j \in I \cap k \langle h^{\pm 1}, f \rangle. $$ There exists $j$
such that $0 \neq \beta e_j = \gamma_j e_j.$ That is, relabelling,
there exists $j$ with $0 \leq j \leq b-1$ and $0 \neq \gamma \in k
\langle h^{\pm 1} \rangle$ such that $\gamma e_j \in I.$ Notice now
that $e_j y = y e_{j-1}.$ It follows that, for all $t = 0, \ldots ,
b-1,$
$$ I \ni y^{n-t} \gamma e_j y^t = y^n \gamma_t e_{j-t} =
(1 - h^{nw'})\gamma_t e_{j-t}, $$
where $0 \neq \gamma_t \in k \langle h^{\pm 1} \rangle.$ Multiplying
these $b$ non-zero elements of $I$ on the left by suitable elements
of the commutative domain $k \langle h^{\pm 1} \rangle,$ we find
that there exists $0 \neq c \in k \langle h^{\pm 1} \rangle$ such
that $ce_{j-t} \in I$ for all $t = 0, \ldots , b-1.$ Adding these
elements, \eqref{yyE3.4.13} follows.

(d) Since $B$ is prime by (c) we can work in the simple artinian
quotient ring $Q(B)$ of $B$. Now $Q(Z(B)) = Z(Q(B)) = k(x)$, and we
see from the definition \eqref{yyE3.4.1} of $x$ and the $k$-basis
\eqref{yyE3.4.3} that
$$\mathrm{dim}_{Q(Z)}(Q(B)) = n'bn = n^2.$$
This proves (d).

(e) Suppose first that $w$ is a unit in $k$. by \cite[Corollary
2.4]{LL} it suffices to prove that the trivial $B$-module $k$ has
projective dimension 1. Noting \eqref{yyE3.4.6} that $y$ is a normal
regular element and setting $D := B/yB,$ $D$ has $k-$basis $\{h^if^j
: 0 \leq i \leq nw' - 1, 0 \leq j \leq b-1\},$ so that
\begin{eqnarray} D \cong k(C_{n'w} \times C_b). \label{resolve}
\end{eqnarray}
it follows that the group algebra $D$ of (\ref{resolve}) is
semisimple Artinian, so we have that $\mathrm{prdim}_{B}(k) = 1$
as required.

Now suppose that $p|w,$ so that $p|w '$. Suppose for a
contradiction that $B$ has finite global dimension. It is clear from
\eqref{yyE3.4.4} that $B$ is a free module over the commutative
algebra $R := k \langle h^{\pm n}, y \rangle,$ and therefore, by
restricting $B$-resolutions to $R$-resolutions we deduce that $R$
has finite global dimension. But from \eqref{yyE3.4.2} we see that, as a
commutative $k$-algebra, $$ R \cong k \langle U^{\pm 1},V :
U^{w'} = 1 - V^n \rangle, $$ and the Jacobian Criterion shows
that $R$ is singular. Thus $B$ also has infinite global dimension.

(f) By \cite[Lemma 2.6]{LWZ}, the left homological integral
$\int^l_B$ of $B$ is the the 1-dimensional right module $B/\langle
y, h-\theta^{-1},f-\theta^{-n'} \rangle.$ Hence the left and right
winding automorphisms associated to $\int^l_B$ are
$$
\Xi_{\pi}^l: \begin{cases} y\mapsto y,\\
h\mapsto \theta^{-1}h ,\\ f\mapsto \theta^{-n'}f
\end{cases}
\quad {\text{and}}\quad
\Xi_{\pi}^r: \begin{cases} y\mapsto \theta^{-(w' + n'i_0)}y,\\
h\mapsto \theta^{-1}h ,\\ f\mapsto \theta^{-n'}f.
\end{cases}
$$
Clearly these automorphisms have order $n,$  whence $io(B)=n$.
Moreover our choice of $i_0$ as a solution of \eqref{yyE3.4.9} shows
that $ G^l_{\pi} \cap G^r_{\pi} = \{1\}$. This proves
(f).

(g) This follows easily from the descriptions in (f) of the
automorphisms.

(h) It's clear that $k\langle h,f \rangle$ is the biggest sub-Hopf
algebra of $B$ which is a group algebra, so this follows from the
linear independence of distinct group-like elements
\cite[3.2.1]{Sw}.

(i) By \eqref{yyE3.4.3}, $z=\sum_{i=0}^{n-1} z_i y^i$ where $z_i\in
A:=k \langle x^{\pm 1},g\rangle=k\langle h^{\pm 1},f\rangle$. Pick
$r$ such that $z_r\neq 0$ and $z_i=0$ for all $r<i\leq n-1$. Hence
$$\begin{aligned}
\Delta(z)&=\sum_{i=0}^r \Delta(z_i) \Delta(y^i)=
\sum_{i=0}^r \Delta(z_i) \Delta(y)^i\\
&=\sum_{i=0}^r \Delta(z_i) (y\otimes g+1\otimes y)^i\\
&=\sum_{i=0}^r \Delta(z_i) (\sum_{s=0}^i y^s\otimes {i\choose s}_{\xi}
g^s y^{i-s}).
\end{aligned}
$$
Since $z$ is a $(a,1)$-primitive,
$$
\Delta(z)=z\otimes a+1\otimes z=(\sum_{i=0}^{r} z_i y^i)
\otimes a+1\otimes (\sum_{i=0}^{r} z_i y^i).$$
If $r\geq 2$, by comparing the terms of $y$-degree $(1,r-1)$ in
two expressions of $\Delta(z)$ and using \eqref{yyE3.4.3}, we have
$$\Delta(z_r)(y\otimes {r\choose 1}_{\xi} g^1 y^{r-1})=0,$$
which implies that $\Delta(z_r)=0$, and hence $z_r=0$ after applying
$\epsilon\otimes 1$. This contradicts the choice of $r$.
Therefore $z=z_0+z_1 y$. By $(a,1)$-primitiveness of $z$,
$$\Delta(z_0)+\Delta(z_1)(y\otimes g+1\otimes y)=
(z_0+z_1y)\otimes a+1\otimes (z_0+z_1y),$$
which implies that
$$\Delta(z_0)=z_0\otimes a+1\otimes z_0,
\quad \Delta(z_1)(1\otimes g)=z_1\otimes a,\quad
\Delta(z_1)=1\otimes z_1.$$
By using the coalgebra axioms several times, we obtain
$$z_1=c_1, \quad a=g, \quad z_0=c_2(1-g)$$
for some $c_1,c_2\in k$, as desired.

(j) By (d) $n=\bar{n}$. By (b), $\phi(\bar{x})$ is either $x$ or
$x^{-1}$. Let $z=\phi(\bar{y})$ and $a=\phi(\bar{g})$. Since $\bar{y}$
is $(\bar{g},1)$-primitive, $z$ is $(a,1)$-primitive. By (i),
$a=g$ and $z =c_1 y+c_2(1-g)$. Since $\{\bar{y},\bar{g}\}$ are
skew-commutative, so is the pair $\{z,a=g\}$. This implies
that $c_2=0$. The relation $y^n=1-g^n$ implies that $c_1=\eta$
is an $n$th root of $1$ and hence $\phi(\bar{y})=\eta y$.
By the relations in \eqref{yyE3.4.12}, we have
$$g^n=\phi(\bar{g}^n)=\phi(1-\bar{x}^{\bar{w}})=1-x^{\pm \bar{w}}$$
and
$$g^n=1-x^w.$$
This forces $w=\bar{w}$ and $\phi(\bar{x})=x$. Finally,
$\xi=\bar{\xi}$ follows from the first relation in
\eqref{yyE3.4.12} and the facts that $\eta y=\phi(\bar{y})$ and
$g=\phi(\bar{g})$.
\end{proof}

\begin{remarks} Suppose that $\gcd (n,w) = b = 1$ in the above construction. Then
the group of group-like elements of $B$ is $\langle h \rangle,$
infinite cyclic, so $B = k \langle h,y \rangle$ is a Hopf algebra
of the class constructed by Liu \cite{Liu}. Indeed it's clear that
all of Liu's examples arise in this way, precisely as those algebras
$B(n,w,\xi)$ with $\gcd (n,w) = 1.$
\end{remarks}

Here is the lemma needed to find a solution $i$ to \eqref{yyE3.4.9}.

\begin{lemma} Let $n$ and $w$ be positive integers with $\gcd (n,w)
= b,$ and write $n = bn', w = bw'.$ There exists $i,$ $0 \leq i \leq
b-1,$ such that $\gcd (n, w' + n'i) = 1.$
\end{lemma}
\begin{proof} Define $i_0 := \prod \{ p : p \textit{ prime, } p|b, p
\nmid w' \},$ with $i_0 = 1$ if there are no primes in the indicated
set. We claim that $\gcd (n, w' + n'i_0) = 1.$ Let $q$ be a prime
divisor of $n.$ If (i) $q|n'$ then $q \nmid w'$, so $q \nmid(w' + n'
i_0).$ Suppose (ii) that $q \nmid n',$ so that $q|b.$ If $q|w'$,
then $q \nmid i_0$ and so $q \nmid n' i_0.$ Therefore $q \nmid (w' +
n' i_0).$ If $q \nmid n',$ so that $q|b,$ but $q \nmid w',$ then $q
| i_0,$ and hence $q \nmid (w' + n' i_0).$ Combining (i) and (ii)
yields the result.
\end{proof}

\section{Classification I: $im(H) = 1$}
\label{yysec4}
\subsection{Primitive versus group-like}
\label{yysec4.1} Having dealt in Proposition \ref{yysec3.1} with
the case where $io (H) =1,$ henceforth we shall always assume that
$$n:=io(H)>1.$$

We begin the classification proof by dealing first with the subclass
of prime affine regular Hopf algebras $H$ of GK-dimension one for
which
\begin{eqnarray} im(H) = 1.
\label{yyE4.1.1}
\end{eqnarray}
In fact, in a small (and, as we shall see eventually, purely
formal) weakening of this assumption, it will be convenient to
assume that \begin{quote} \textit{there exists an algebra map}
$\pi: H\to k$ \textit{satisfying both Hypotheses \ref{yysec2.5}
and the condition}
\begin{eqnarray} |G^l_{\pi}/(G^l_{\pi}\cap G^r_{\pi})|= 1.
\label{yyE4.1.2}
\end{eqnarray}\end{quote}
Notice that \eqref{yyE4.1.2} implies that $G^l_{\pi}=G^r_{\pi}$ and
hence $H^l_\pi=H^r_\pi=H_\pi= \widetilde{H_\pi}$, the latter being a
subHopf algebra of $H$, by Proposition \ref{yysec2.1}(h). In
particular, cocommutative Hopf algebras satisfy \eqref{yyE4.1.1} and
\eqref{yyE4.1.2}; conversely, it follows from the theorem we prove
below that every Hopf algebra $H$ satisfying Hypotheses
\ref{yysec2.5} and \eqref{yyE4.1.2} is cocommutative.


More precisely, we shall prove

\begin{theorem}
Suppose $(H,\pi)$ satisfies Hypotheses \ref{yysec2.5} and
\eqref{yyE4.1.2}. Then $H$ is isomorphic as a Hopf algebra either to
the Taft algebra $H(n,0,\xi)$ of Example \ref{yysec3.3}, or to the
dihedral group algebra $k\mathbb{D}$ of Example \ref{yysec3.2}. As a
consequence, $H$ is cocommutative and satisfies \eqref{yyE4.1.1}.
\end{theorem}

The conclusion that cocommutativity follows from \eqref{yyE4.1.2} is
\emph{not} valid when the GK-dimension is bigger than one - for
example $\mathcal{U}:=\mathcal{U}_{\epsilon}(\mathfrak{sl}_2)$ is an
affine regular domain and a Hopf algebra, which is PI with
GK-dimension 3. In this case, $io(\mathcal{U}) = im(\mathcal{U}) =
1,$ but $\mathcal{U}$ is not cocommutative. But we don't seem to
know what happens in GK-dimension 2.

The dichotomy appearing in the following definition will feature
heavily throughout the rest of the paper.

\begin{definition}
Let $H$ be a prime affine regular Hopf algebra of GK-dimension one.
Assume Hypotheses \ref{yysec2.5} and \eqref{yyE4.1.2}. By
Proposition \ref{yysec2.1}(h), together with Theorem
\ref{yysec2.5}(f) and Examples \ref{yysec3.1}, $H_\pi$ is isomorphic
as a Hopf algebra either to $k[x]$ or to $k[x^{\pm 1}]$.
\begin{enumerate}
\item
$H$ is called {\it primitive} if $H_\pi=k[x]$.
\item
$H$ is called {\it group-like} if $H_\pi=k[x^{\pm 1}]$.
\end{enumerate}
\end{definition}

We shall see from our final result that this definition is
independent of the choices of $\pi$ under Hypotheses \ref{yysec2.5}.

\subsection{Primitive case}
\label{yysec4.2}

\begin{proposition}
Suppose $(H,\pi)$ satisfies Hypotheses \ref{yysec2.5} and
\eqref{yyE4.1.2}. If $H$ is primitive, then $H$ is isomorphic as a
Hopf algebra to the Taft algebra $H(n,0,\xi)$ of Example \ref{yysec3.3}.
\end{proposition}

\begin{proof} Since $\widehat{G^l_{\pi}}={\mathbb Z}_n=\{0,1,
\cdots,n-1\}$, we write $H^l_{\pi,i}=H^l_i$ for all $i\in
\widehat{G^l_{\pi}}$. Similarly define $H^r_i$. By \eqref{yyE4.1.2},
we have
$$H^l_0=H^l_\pi=H^r_{\pi}=H^r_0$$
which is also equal to $H_0$ (and $\widetilde{H_0}$).

Since $H$ is primitive, $H_0 = k[x]$ and it is a Hopf subalgebra
by Proposition \ref{yysec2.1}(h), after a linear change of variable
we may assume
$$\Delta(x)=x\otimes 1 + 1\otimes x , \quad S(x) = -x
\textit{ and } \epsilon (x) = 0. $$ By Proposition \ref{yysec2.2},
$H=\oplus_{i\in {\mathbb Z}_n} H^l_i,$ with each $H^l_i$ an
invertible module over $H^l_0=k[x]$. Hence, for each $i$,
$H^l_i=u_ik[x]=k[x]v_i$ for some $u_i,v_i\in H^l_i$. Fix $\alpha,
\delta \in k[x]$ so that $u_i = \alpha v_i$ and $v_i = u_i \delta.$
Then $k[x]u_i \subseteq u_ik[x],$ so that $u_i = \alpha u_i \delta =
u_i \beta \delta$ for some $\beta \in k[x].$ Thus $\delta$ is a unit
of $k[x]$ and so $H^l_i = u_i k[x] = k[x]u_i.$ Now the strong
grading condition implies that $u_iu_{-i}$ is a unit element in
$H^l_0=k[x]$. Hence each $u_i$ is a unit element. Set $g=u_1$. Then
we may take $u_i:=g^i$ for all $i,$ $0\leq i<n,$ so that $H$ is
generated as a $k-$algebra by $x$ and $g.$ Now $g^n$ equals a
non-zero scalar in $H^l_0=k[x]$, so, since $k$ contains primitive
$n$th roots thanks to Hypotheses \ref{yysec2.5}(d), we can adjust
our choice of $g$ if necessary so that $g^n=1.$ This in turn forces
$\epsilon (g)$ to be an $n$th root of 1 in $k,$ so, possibly after
another adjustment, we can assume that
\begin{eqnarray}
\epsilon (g) = 1.
\label{yyE4.2.1}
\end{eqnarray}

Fix $i$, $1 \leq i \leq n-1$. Since, by \eqref{yyE4.1.2}, $G^l_\pi
=G^r_\pi$, $H^l_i=H^r_j$ for some $j$, $1 \leq j \leq n-1$.
That is, there is a permutation $i\mapsto i'$ of
$\{1, \ldots , n-1 \}$ such that $H^l_i=H^r_{i'}$.

In particular, by Proposition and Remark \ref{yysec2.2}(a),
$$\Delta(g)\in (H^l_1\otimes H)\cap (H \otimes H^r_{1'}) =
H^l_1\otimes H^r_{1'} = H^l_1\otimes H^l_1.$$ So we can write
$$\Delta(g)=(g\otimes g)q$$
for some $q\in H_0\otimes H_0 = k[x] \otimes k[x]$. Since $g$ is
invertible, so is $q$, and so $q \in k \setminus \{ 0 \}.$  Now
(\ref{yyE4.2.1}) and the fact that $m_A\circ (id \otimes
\epsilon)\circ \Delta = id$ force $q=1$. Therefore $g$ is
group-like, and so $S(g)=g^{-1}.$

Now $H^l_1 = k[x]g = gk[x],$ so conjugation by $g$ yields an
algebra automorphism of $k[x].$ Thus there are $a,b \in k$ with
$b \neq 0$ such that
$$xg=g(a+bx).$$
Applying $\epsilon$ to this shows that $a=0;$ and, since $g^n=1$,
we have $b^n=1$. Therefore $H$ is a factor of the algebra
$C:=k\langle x,g : xg=bgx,g^n=1 \rangle ,$ and both algebras are
free of rank $n$ over $k[x]$. Thus $H = C,$ and primeness of $H$
entails that $b$ is a primitive $n$th root of 1. This completes
the proof.
\end{proof}

\subsection{Group-like case}
\label{yysec4.3}

\begin{proposition}
Suppose $(H,\pi)$ satisfies Hypotheses \ref{yysec2.5} and
\eqref{yyE4.1.2}. If $H$ is group-like, then $H$ is isomorphic as a
Hopf algebra to the dihedral group algebra $k{\mathbb D}$ of Example
\ref{yysec3.2}.
\end{proposition}

\begin{proof} Since $H$ is group-like, $H_0=k[x^{\pm 1}]$.
The first three paragraphs of the proof of Proposition
\ref{yysec4.2} apply \emph{mutatis mutandis} to the group-like
case. We find that, after adjustments, there is an invertible
element $g \in H^l_1$ with $\varepsilon (g) = 1$ such that $H^l_i
= g^ik[x^{\pm 1}]$ for $i = 0, \dots , n-1.$ Moreover, there
exists $j \in \mathbb{Z}$ with \begin{eqnarray} g^n \, = \, x^j.
\label{late} \end{eqnarray} And there is a permutation $i\mapsto
i'$ of $\{1, \ldots , n-1\}$ such that $H^l_i=H^r_{i'}$. As in the
previous subsection,
$$\Delta(g)=(g\otimes g)q$$
for an invertible element $q$ of $H_0 \otimes H_0.$ Applying the
Hopf axioms as before, we deduce that $q = 1$, so that $g$ is a
group-like element. Therefore $H$ is generated as a $k-$algebra by
the two group-like elements $g$ and $x$, so $H$ is a group
algebra. Now it is easy to deduce that, if $n \neq 0$ then $H
\cong k\mathbb{D}$ is the only possibility, as in
\cite[Proposition 8.2(b)]{LWZ}.
\end{proof}

\section{The center}
\label{yysec5}
This section is preparatory for $\S 6,$ where we
deal with the case  $im(H)=io(H)$ of the classification.

\subsection{Abelian strong gradings}
\label{yysec5.1} We begin with some further analysis of an algebra
strongly graded by a finite abelian group, with commutative
identity component, a set-up we have been led to in \S
\ref{yysec2.5}. Recall that, if $Z$ is an Ore domain, then the
\emph{rank} of a $Z$-module $M$ is defined to be the
$Q(Z)-$dimension of $Q(Z)\otimes_Z M$; the rank of a right
$Z$-module is defined analogously. The PI-degree of an algebra $R$
has been recalled in (\ref{yysec2.3}).

\begin{proposition}
Let $D$ be a $k$-algebra and $G$ a finite abelian group acting
faithfully on $D$. Suppose that the order $n$ of $G$ is a unit
in $k$, so that $D$ is $\widehat{G}-$graded,
\begin{eqnarray}
D=\oplus_{\chi\in \widehat{G}} D_{\chi}.
\label{yyE5.1.1}
\end{eqnarray}
Assume that
\begin{enumerate}
\item[(i)]
\eqref{yyE5.1.1} is a strong grading, and
\item[(ii)]
$D_0$ is a commutative domain, with field of fractions $Q_0$.
\end{enumerate}
Then the following hold.
\begin{enumerate}
\item
$D$ is a semiprime Goldie PI-algebra, with
$\PIdeg(D) \leq n.$ Every nonzero homogeneous
element is a regular element of $D$.
\item
There is an action of $\widehat{G}$ on $Q_0$. Denote the
automorphism of $Q_0$ corresponding to $\chi \in
\widehat{G}$ by $\kappa_{\chi}.$ Then, for every
non-zero element $x$ of $D_{\chi}$,
\begin{eqnarray}
xa=\kappa_{\chi}(a) x
\label{yyE5.1.2}
\end{eqnarray}
for all $a\in Q_0$. For all $\chi \in \widehat{G},$
$\kappa_{\chi}(D_0) \subseteq D_0;$ that is, $\kappa_{\chi}$
restricts to an automorphism of $D_0$.
\item
Let $K:=\{\kappa_{\chi}\;|\; \chi\in \widehat{G}\} \subseteq
\mathrm{Aut}(D_0)$. If $L$ is a cyclic subgroup of the kernel of
the homomorphism $\kappa : \widehat{G}\to K : \chi \mapsto
\kappa_{\chi}$, then
$$ \PIdeg(D) \leq |\widehat{G}/L|.$$
\item
Consider the following statements:
\begin{enumerate}
\item[(1)] $\PIdeg(D) = n$. \item[(2)] $\widehat{G}$
acts faithfully on $Q_0.$ \item[(3)] $\widehat{G}$ acts faithfully
on $Q_0$ and the center $Z(D) \subseteq D_0.$ \item[(4)] $D$ is
prime.
\end{enumerate}
Then (1), (2) and (3) are equivalent, and they imply (4).
(4) does not imply the others. When (1,2,3) hold,
\begin{eqnarray}
Z(D) = \{a \in D_0 | \kappa_{\chi}(a) = a \; \forall \chi \in
\widehat{G} \}.
\label{yyE5.1.3}
\end{eqnarray}
\end{enumerate}
\end{proposition}

\begin{proof}
Thanks to hypotheses (i) and (ii) $D$ is a torsion free module of
finite rank over a commutative ring, and so $D$ is a PI-algebra.
Since the nilpotent radical $N$ of $D$ is $G-$invariant, it is
$\widehat{G}-$graded. Therefore, by strong grading (i), $N = (N
\cap D_0)D$, so (ii) forces $N=0$. We claim that
\begin{eqnarray} D \textit{ is a right Goldie ring. }
\label{yyExtra}\end{eqnarray} For, being strongly graded, $D$ is a
projective (hence torsion-free) right $D_0-$module. Thus, as right
$D_0-$module, $D$ embeds in the finite dimensional $Q(D_0)-$vector
space $DQ(D_0),$ and (\ref{yyExtra}) follows easily. Therefore
Goldie's theorem \cite[Theorem 2.3.6]{MR} ensures that $D$ has a
semisimple artinian quotient ring $Q$. The $G-$action on $D$ extends
to a $G-$action on $Q$, so that $Q$ is a $\widehat{G}$-graded
algebra containing $D$ as a graded subalgebra. Since $D$ is strongly
$\widehat{G}$-graded, so is $Q$, $Q = \oplus_{\chi \in
\widehat{G}}Q_{\chi}$. In particular, each non-zero element of
$Q_{\chi}$ is a unit, and no nonzero element of $D_{\chi}$ is a zero
divisor.

For each $\chi \in \widehat{G},$ $D_{\chi}$ is an invertible
$D_0$-module, so $Q_0D_{\chi}=Q_0u_{\chi}$ for any nonzero element
$u_{\chi}\in D_{\chi}$. Thus $Q_0D=\oplus_{\chi \in \widehat{G}}
Q_0u_{\chi}$ is a finite dimensional left $Q_0$-vector space.
Since right multiplication of this space by any non-zerodivisor of
$D$ induces an invertible map, $Q_0D=Q$. Similarly, $DQ_0=Q$.
Therefore
$$Q=\oplus_{\chi\in \widehat{G}} Q_0u_{\chi}=
\oplus_{\chi\in \widehat{G}} u_{\chi} Q_0$$ as graded spaces. This
also shows that $D_0-\{0\}$ is an Ore set of $D$ and the associated
localization of $D$ is $Q$. Thus we may assume whenever convenient
in the rest of the proof that $D_0=Q_0$ is a field and $D=Q$ is
semisimple artinian.

(a) Since $D$ is a subring of $\End(Q_{Q_0})=\End (Q_0^{\oplus
n})\cong M_n(Q_0)$ where $n=|\widehat{G}|$, $\PIdeg(D)
\leq n.$

(b) For $a \in Q_0,$ define
$\kappa_{\chi}(a)=u_{\chi}au_{\chi}^{-1}$. Then $\kappa_{\chi}$ is
an automorphism of $Q_0$ independent of the choice of a nonzero
element $u_{\chi} \in Q_{\chi}=Q_0u_{\chi}$. The map $\kappa :
\chi\to \kappa_{\chi}$ induces a surjective group homomorphism from
$\widehat{G}$ to $K=\{\kappa_{\chi}\;|\; \chi\in \widehat{G}\}$.

Let $a \in D_0$ and $\chi \in \widehat{G}.$ From (\ref{yyE5.1.2})
with $x$ in $D_{\chi}$ we see that $$\kappa_{\chi}(a)D_{\chi}
\subseteq D_{\chi}.$$ Multiplying this inclusion on the right by
$D_{-\chi}$ and appealing to strong grading, we deduce that
$\kappa_{\chi} (a)D_0 \subseteq D_0.$ Thus $\kappa_{\chi}(a) \in
D_0$ as claimed.

(c) Let $L$ be any cyclic subgroup of the kernel of the map from
$\widehat{G}$ to $K$ and let $\chi_L$ be the generator of $L$.
Let $t$ be  the order of $P:=\widehat{G}/L$. For each coset $w$
in $P$, set
$$C_w=\oplus_{ \chi \in w} Q_{\chi}.$$
Then $D=\oplus_{w \in P}C_w,$ and $C_0$ is the commutative
subalgebra of $D$ generated by $Q_0$ and $u_{\chi_L}$.
(Note that $u_{\chi_{L}}$ commutes with $Q_0$ because $\chi_L$
is in the kernel of $\widehat{G}\to K$.) For each $w \in P$
choose $\chi \in w$ and $0 \neq u_{\chi} \in Q_{\chi}$.
Define $u_w := u_{\chi},$ so that
$$D=\oplus_{w\in P}C_w=
\oplus_{w \in P}u_wC_0= \oplus_{i\in P}C_0u_w.$$
Therefore $D$ is strongly $P$-graded and it is a subring of
$\End(D_{C_0})\cong M_t(C_0)$. Hence, $\PIdeg(D) \leq t$.

(d) Suppose (1), that $\PIdeg(D)=n$. Then (2) follows
from (c). Suppose (2), so that the map from $\widehat{G}$ to $K$
is an isomorphism. The center of $D$ is $\widehat{G}$-graded, and
if $0 \neq x\in D_{\chi} \cap Z(D) $ for some $0 \neq \chi \in
\widehat{G}$ then (\ref{yyE5.1.2}) is contradicted. Therefore $Z(D)
\subseteq Q_0,$ and (\ref{yyE5.1.3}) is clear. Thus (2) implies
(3).

Suppose now that (3) holds. Since $K$ acts on $Q_0$ faithfully,
$$Q_0=\oplus_{\eta\in \widehat{K}} (Q_0)_{\eta},$$
where, by Lemma \ref{yysec2.2}, each $(Q_0)_{\eta}\neq
0$. By hypothesis (3) and (\ref{yyE5.1.3}), $Z(Q)=(Q_0)_0$. Thus
$\rank_{Z(Q)}Q_0\geq |K|=n$. Since $Q$ is strongly
$\widehat{G}$-graded, $\rank_{Q_0}Q=n$. Combining these we have
$n^2 \leq \rank_{Z(Q)}Q;$ thus, by (a) and the definition of the
PI-degree, $\PIdeg(D) = \PIdeg(Q) = n,$
showing that (3) implies (1).

Assume now that (1), (2) and (3) hold. Then $Z(D)\subset D_0$
is a domain. This implies that $D$ is prime.

Finally, take $D= \mathbb{C}[x^{\pm 1}]$ with the group $G$ of
order 2 acting by $x \mapsto -x$, to see that (4) does not
imply (1).
\end{proof}

\begin{corollary}
Assume the hypotheses of Proposition \ref{yysec5.1} above and
suppose that
$$\PIdeg(D)=n=|G|.$$
Let $K$ be a subgroup of $\widehat{G}$ and $B$ the subalgebra
$\oplus_{\chi\in K} D_{\chi}$
\begin{enumerate}
\item
The restriction of the automorphisms in $G$ to $B$ induces a group
homomorphism $f:G\to \mathrm{Aut}_{k\text{-}\mathrm{alg}}(B)$ with
image $L$ such that $\widehat{L}=K.$
\item
$B$ is prime with PI-degree $|K|$.
\end{enumerate}
\end{corollary}

\begin{proof} (a) It is clear that $G$ acts in $B$. Therefore $L
=G/ker f$ acts on $B$ faithfully. By the hypotheses of
Proposition \ref{yysec5.1}, $|L|$ is a unit in $k$. Hence
$B=\oplus_{\chi\in \widehat{L}} B_{\chi}$. Since every character
of $L$ is a character of $G$, $B_{\chi}= D_{\chi}$ and
hence
$$B=\oplus_{\chi\in \widehat{L}} B_{\chi}=
\oplus_{\chi\in \widehat{L}} D_{\chi}.$$
By Lemma \ref{yysec2.2}, each $D_{\chi}\neq 0$ for all
$\chi\in \widehat{L}$. Thus $K=\widehat{L}$.

(b) Since $\widehat{L}=K$ is a subgroup
of $\widehat{G}$, it acts on $D_0$ faithfully.
By Proposition \ref{yysec5.1}(d) above, $B$ is prime and
$\PIdeg(B)=|K|=|L|$.
\end{proof}

\subsection{The center}
\label{yysec5.2} We return for the remainder of $\S 5$ to the study
of a prime affine regular Hopf $k-$algebra $A$ of GK-dimension one,
with $io(A) = n$. We assume Hypotheses \ref{yysec2.5}(a,b,c,d) (but
not (e)). We fix any algebra homomorphism $\pi: A\to k$ of order
$n=io(A)$ as in Hypotheses \ref{yysec2.5}(a,b,c,d) and subgroups
$G^l_{\pi}, G^r_{\pi},G_{\pi}$ of
$\mathrm{Aut}_{k-\mathrm{alg}}(A)$, with corresponding fixed rings
$A^l_0$, $A^r_0$ and $A_0,$ as in $\S 2$. There are decompositions
$$A=\oplus_{i\in \widehat{G^l_{\pi}}}A^l_i=
\oplus_{i\in \widehat{G^r_{\pi}}}A^r_i,$$ which are strong gradings
thanks to Proposition \ref{yysec2.2}. Let $Q^l_0$ and $Q^r_0$ denote
the fields $Q(A^l_0)$ and $Q(A^r_0)$ respectively, and let $Q_0 =
Q(A_0)$. As before, $Q$ will denote the simple artinian ring $Q(A)$.
We use $\kappa^l_{\chi}$ and $\kappa^r_{\chi}$ to denote the
automorphisms associated to an element $\chi$ from $G^l_{\pi}$ or
$G^r_{\pi}$ respectively, as defined in \eqref{yyE5.1.2}. These form
subgroups $K^l_{\pi}:=\{\kappa^l_{\chi}: \chi \in G^l_{\pi}\}$ and
$K^r_{\pi}:=\{\kappa^r_{\chi}: \chi \in G^r_{\pi}\}$ of
$\mathrm{Aut}(A^l_0)$ and $\mathrm{Aut}(A^r_0)$ respectively. We aim
to compare the actions of two groups on $A^l_0$ - first, since
$G^r_{\pi}$ and $G^l_{\pi}$ commute by Proposition
\ref{yysec2.1}(d), $G^r_{\pi}$ acts on $A^l_0$. Denote by $\rho^l$
the resulting map from $G^r_{\pi}$ to $\mathrm{Aut}(A^l_0),$ bearing
in mind that $\rho^l$ may not be a monomorphism. The second group
acting on $A^l_{0}$ is $\widehat{G^l_{\pi}},$ acting through the map
$\kappa^l$ to $K^l_{\pi}$ as defined in Proposition
\ref{yysec5.1}(b). By Hypotheses \ref{yysec2.5}(b) and Proposition
\ref{yysec5.1}(d), $\kappa^l$ is an isomorphism. As we now show,
these two group actions are closely related. Let's write $$P^l_{\pi}
:= \langle \rho^l (G^r_{\pi}), \kappa^l(G^l_{\pi}) \rangle = \langle
\rho^l (G^r_{\pi}), K^l_{\pi}) \rangle \subseteq
\mathrm{Aut}(A^l_0);$$ notice that we can (and will) mimic this
set-up, starting instead from $A^r_0$  and obtaining a subgroup
$$P^r_{\pi}:= \langle \rho^r (G^l_{\pi}), \kappa^r(G^r_{\pi}) \rangle
= \langle \rho^r (G^l_{\pi}), K^r_{\pi}) \rangle$$ of
$\mathrm{Aut}(A^r_0).$

\begin{lemma} Assume Hypotheses \ref{yysec2.5}(a,b,c,d).
Then $Z(A)\subseteq A_0$.
\end{lemma}

\begin{proof} By Hypotheses \ref{yysec2.5}(b),
$\PIdeg(A) = n,$ and so, by Proposition
\ref{yysec5.1}(d), $Z(A)\subset A^l_0$. Similarly,
$Z(A)\subset A^r_0$. Since $A_0 = A^l_0 \cap A^r_0,$
the assertion follows.
\end{proof}

\begin{proposition} Assume Hypotheses \ref{yysec2.5}(a,b,c,d).
Retain the notation stated so far in \S \ref{yysec5.2}.
\begin{enumerate}
\item $P^l_{\pi}$ is abelian, and in fact $P^l_{\pi} = K^l_{\pi}.$
\item The following statements are equivalent:
\begin{enumerate}
\item[(1)]
$G^l_\pi\cap G^r_\pi=\{1\}$.
\item[(2)]
$|G_{\pi}| = n^2.$
\item[(3)]
$\rho^l (G^r_{\pi}) = K^l_{\pi} = P^l_{\pi}.$
\item[(4)]
$Z(A) = A_0.$
\item[(5)]
$\rho^r (G^l_{\pi}) = K^r_{\pi} = P^r_{\pi}.$
\end{enumerate}
\end{enumerate}
\end{proposition}

\begin{proof} (a) Let $\chi\in \widehat{G^l_{\pi}}$. Then for
all $a \in A^l_0$ and all $x \in A^l_{\chi} \setminus \{0 \},$
\begin{equation}
xa=\kappa_{\chi}(a)x. \label{yyE5.2.1}
\end{equation}
Since $G^l_{\pi}$ commutes with $G^r_{\pi}$, $A^l_{\chi}$ is
$G^r_{\pi}$-stable. So we may take $x$ in \eqref{yyE5.2.1} to
be an eigenvector for $G^r_{\pi},$ say $g(x)=\eta(g)x$ for all
$g\in G^r_{\pi}$. Applying $g$ to \eqref{yyE5.2.1} gives
$$\eta(g)xg(a)=g(\kappa_{\chi}(a))\eta(g) x$$
for all $a \in A^l_0.$ This proves that
$g\kappa_{\chi}=\kappa_{\chi}g,$ and the first claim follows.

Suppose now for a contradiction that $K^l_{\pi} \varsubsetneq
P^l_{\pi}.$ Since $n$ is a unit in $k$, so is $|P^l_{\pi}|$ by
the first claim, and so
\begin{eqnarray}
Q^l_0=\oplus_{\eta\in\widehat{P_{\pi}}}(Q^l_0)_{\eta}
\label{yyE5.2.2}
\end{eqnarray}
with each term $(Q^l_0)_{\eta}\neq 0$ by Lemma \ref{yysec2.2}.
Now $Z(A)\subseteq A_0 \subseteq Q^l_0$ by the above lemma;
moreover $Z(A)$ is fixed by the actions of $G^r_{\pi}$ and
$K^l_{\pi}$ on $Q^l_0$, the former because $Z(A) \subseteq
A^r_0 = A^{G^r_{\pi}}$, the latter because $K^l_{\pi}$ acts
by inner automorphisms of $Q(A)$. Thus, with respect to the
decomposition \eqref{yyE5.2.2}, $Z(A) \subseteq (Q^l_0)_{0}$.
So $\rank_{Z(A)}(A^l_0)>n$ and therefore $\rank_{Z(A)}(A)>n^2$,
contradicting the fact that $\PIdeg(A) = n$.
Therefore $P^l_{\pi}=K^l_{\pi}$.

(b) $(1) \Leftrightarrow (2):$ This is clear from the definition
of $G_\pi$.

$(2)\Rightarrow (4):$ Suppose that $|G_{\pi}| = n^2 ,$ so that
$$G_{\pi} = G^l_{\pi} \times G^r_{\pi},$$ with this group of
order $n^2$ acting faithfully on $A$. Define $A_{ij}:=A^l_i\cap
A^r_j,$ for $0 \leq i,j \leq n-1,$ so, in particular, $A_{00} =
A_0$. Hence there is a decomposition \begin{eqnarray}A=\oplus_{(i,j)
\in \widehat{G_{\pi}}}A_{ij}. \label{gott} \end{eqnarray}  Since
(\ref{gott}) is a refinement of the strong grading $A =
\bigoplus_{i=0}^{n-1}A^l_i ,$ for which the homogeneous elements are
regular by Proposition \ref{yysec5.1}(a), $A$ is a
$\widehat{G_{\pi}}-$domain. Therefore Lemma \ref{yysec2.2} applies,
and we conclude that $A_{ij}\neq 0$ for every $(i,j)\in
\widehat{G_{\pi}}$. Therefore $A_{ij}$ is a torsion free
$Z(A)-$module of rank at least one. However from Hypotheses
\ref{yysec2.5}(b) and the definition of PI-degree, we know that
$\mathrm{rank}_{Z(A)}(A) = n^2$. Therefore
$\mathrm{rank}_{Z(A)}(A_{00}) = 1$. In other words,
$A_{00}Q(Z(A))=Q(Z(A))$, so $A_{00}$ is central and the inclusion of
Lemma \ref{yysec5.2} is an equality.

$(4)\Rightarrow (3):$ Suppose that $Z(A) = A_0 = (A^l_0)^{G^r_{\pi}}$,
so that $Z(Q) = Q_0 = (Q^l_0)^{\rho^l (G^r_{\pi})}$. Now the
$\widehat{G^r_{\pi}}-$grading of $Q$ is strong, by Proposition
\ref{yysec2.2} -- equivalently, the non-zero elements of each
homogeneous component are units in $Q$. Using this equivalent
formulation it is immediately clear that $Q^l_0$ is strongly
$\rho^l (G^r_{\pi})-$graded, so that
\begin{eqnarray}
\mathrm{dim}_{Q_0}(Q^l_0) = |\rho (G^r_{\pi})|.
\label{yyE5.2.3}
\end{eqnarray}
However, since $\PIdeg(A) = n$ by Hypotheses \ref{yysec2.5}(b),
$$
n^2 = \dim_{Z(Q)}(Q) = \dim_{Q^l_0}(Q)(\dim_{Z(Q)}(Q^l_0))
= n.\dim_{Z(Q)}(Q^l_0) = n.\dim_{Q_0}(Q^l_0).
$$
Hence $\dim_{Q_0}(Q^l_0)=n$, so \eqref{yyE5.2.3} shows that
$\rho^l$ is a monomorphism, as required.

$(3)\Rightarrow (1):$ Suppose that $\rho^l (G^r_{\pi}) = K^l_{\pi}$.
Recall that $K^l_{\pi}$ acts faithfully on $Q^l_0$, by Proposition
\ref{yysec5.1}(d). Thus our hypothesis says that $G^r_{\pi}$ is
acting faithfully on $A^l_0 = A^{G^l_{\pi}}$. It follows that
$G^l_{\pi} \cap G^r_{\pi} = \{1\}$.

$(4)\Leftrightarrow (5):$ This follows from the equivalence of (4)
and (3), since (4) is left-right invariant.

\end{proof}

Recall that $\widetilde{A_0}=\widetilde{A_\pi}$ is defined in
Proposition \ref{yysec2.1}(g).

\begin{theorem}
Assume Hypotheses \ref{yysec2.5}(a,b,c,d). Suppose that the algebra
homomorphism $\pi$ in Hypotheses \ref{yysec2.5} satisfies the
condition $G^l_{\pi} \cap G^r_{\pi} = \{1\}$. Then
$A_0=\widetilde{A_0}$ - that is, $A_0 = Z(A)$ is a Hopf subalgebra
of $A$.
\end{theorem}

\begin{proof}
By Proposition \ref{yysec5.2}(b) above, $Z(A)=A_0$. It remains to show
that $A_0$ is a Hopf subalgebra of $A.$ Let $a \in Z(A).$ Since
$S(A_0)=A_0$ by Proposition \ref{yysec2.1}(e), we need only to
show that
\begin{eqnarray}
\Delta(a) := \sum a_1 \otimes a_2 \in Z(A)\otimes Z(A).
\label{yyE5.2.4}
\end{eqnarray}
There is a decomposition
$$A=\oplus_{(\chi,\eta)\in \widehat{G_{\pi}}} A_{\chi\eta}$$
where $A_{\chi\eta}=A^l_{\chi}\cap A^r_{\eta}$. The argument to
prove $(2) \Rightarrow (4)$ of Proposition \ref{yysec5.2}(b) shows
that
$$
A_{\chi\eta} \neq 0 \quad \forall \; \chi \in
\widehat{G^l_{\pi}},\eta \in \widehat{G^r_{\pi}}.
$$
Now fix $\eta \in \widehat{G^r_{\pi}},$ and let $0\neq x\in
A_{0,\eta}$. By Proposition \ref{yysec2.1}(c),
\begin{eqnarray}
\Delta(x)\in A^l_0\otimes A^r_{\eta}\subseteq
Q^l_0\otimes Q^r_{\eta}.
\label{yyE5.2.5}
\end{eqnarray}
For all $x'\in A^r_{\eta}$ and $a'\in A^r_0$, the definition of
$\kappa^r_{\eta}$ in Proposition \ref{yysec5.1}(b) yields
\begin{eqnarray}
x' a'=\kappa^r_{\eta}(a') x'.
\label{yyE5.2.6}
\end{eqnarray}
The left hand entries of $\Delta (a)$ and of $\Delta (x)$ are in
the commutative ring $A^l_0$, and so commute. Applying also
\eqref{yyE5.2.6} to the right hand entries of $\Delta(x),$ we
find that
\begin{eqnarray}
\Delta(x)\Delta(a)=(Id\otimes \kappa^r_{\eta})(\Delta(a))\Delta(x).
\label{yyE5.2.7}
\end{eqnarray}
Moreover, $a \in Z(A)$ and $\Delta$ is an algebra homomorphism, so
that
\begin{eqnarray}
\Delta(a)\Delta(x)=\Delta(x)\Delta(a)
\label{yyE5.2.8}
\end{eqnarray}
Next, we claim that \begin{eqnarray}\Delta(x) \textit{ is a not a
zerodivisor in } A \otimes A. \label{noh} \end{eqnarray}  In view of
\eqref{yyE5.2.5}, it suffices to show that $Q^l_{\chi}\otimes
Q^r_{\eta} \setminus \{ 0 \}$ consists of nonzerodivisors in
$Q\otimes Q$. As in the proof of Proposition \ref{yysec5.1},
$Q^l_{\chi}=Q^l_0 u_{\chi}$ and $Q^r_{\eta}=Q^r_0 u_{\eta}$ with
$u_{\chi}$ and $u_{\eta}$ invertible in $Q,$ so it is only necessary
to show that the elements of $Q^l_0\otimes Q^r_0 \setminus \{ 0 \}$
are not zero divisors in $Q \otimes Q.$  But $Q^l_0\otimes Q^r_0$ is
a homogeneous component of the strongly $\widehat{G^l_{\pi}}\times
\widehat{G^r_{\pi}}$-graded algebra $Q \otimes Q$, whose homogeneous
components have the form $Q^l_{\alpha} \otimes Q^r_{\beta} =(Q^l_0
\otimes Q^r_0)(u_{\alpha} \otimes u_{\beta}),$ with $u_{\alpha}
\otimes u_{\beta}$ a unit. Thus it suffices to show that
$Q^l_0\otimes Q^r_0$ is a commutative domain, which follows from the
fact that $k$ is algebraically closed, \cite{V}. This proves
(\ref{noh}). With \eqref{yyE5.2.7} and \eqref{yyE5.2.8}, this
implies that
\begin{eqnarray}
\Delta(a)=(Id \otimes \kappa^r_{\eta})(\Delta(a)).
\label{yyE5.2.9}
\end{eqnarray}
Now all the $a_2$ terms in \eqref{yyE5.2.4} are in $A^r_0$ by
Proposition \ref{yysec2.1}(c). Thus it follows from \eqref{yyE5.2.9}
that each $a_2$ belongs to $(A^r_0)^{\kappa^r_{\eta}}.$ Allowing
$\eta$ to vary through $\widehat{G^r_{\pi}},$ we see from
\eqref{yyE5.1.3} in Proposition \ref{yysec5.1}(d) that all the terms
$a_2$ in \eqref{yyE5.2.4} belong to $Z(A)$. By symmetry, all the
$a_1$ terms are in $Z(A)$ also. This completes the proof.
\end{proof}

\section{Classification II: $im(H)=io(H)>1$}
\label{yysec6}
\subsection{Hypotheses and the main theorem}
\label{yysec6.1} We return in \S \ref{yysec6} to the main focus of
this paper, namely the structure of an affine prime regular Hopf
algebra $H$ of GK-dimension one, over the algebraically closed field
$k$. Recall the definitions of the integral order of $H$, $io(H)$,
and the integral minor of $H$, $im(H)$, in \S \ref{yysec2.3} and \S
\ref{yysec2.4}. As noted in Theorem \ref{yysec2.5}, $io(H)$ equals
the PI degree of $H$, and so in particular is finite, and $im(H)$ is
a positive integer dividing $io(H)$. As in \S \ref{yysec4.1} it will
be convenient for later steps in the proof to weaken our hypotheses
slightly; namely, we'll assume Hypotheses \ref{yysec2.5}, with the
homomorphism $\pi$ of Hypotheses \ref{yysec2.5}(b) satisfying the
condition:
\begin{eqnarray}
n:= |G^l_\pi|>1 \quad \textit{and}\quad G^l_\pi\cap G^r_\pi=\{1\}
\label{yyE6.1.1}.
\end{eqnarray}
The main result of this section, which we now state, shows that
\eqref{yyE6.1.1} is equivalent to $io(H)=im(H):$

\begin{theorem}
Suppose $(H,\pi)$ satisfies Hypotheses \ref{yysec2.5} and
\eqref{yyE6.1.1}. Then $H$ is isomorphic as a Hopf algebra either to
the Taft algebra $H(n,t,\xi)$ of Example \ref{yysec3.3} with
$\gcd(t,n)=1$, or the algebra $B(n,w,\xi)$ in Example
\ref{yysec3.4}. As a consequence, $io(H)=im(H)$.
\end{theorem}

\subsection{Generators and relations}
\label{yysec6.2} Fix a primitive $n$th root $\zeta$ of 1 in $k,$ and
define $\chi \in \widehat{G^l_{\pi}}$ and $\eta \in
\widehat{G^r_{\pi}}$ by setting
$$\chi(\Xi^l_{\pi}):=\zeta \quad \textit{and}
\quad \eta(\Xi^r_{\pi})=\zeta.$$
Thus $\widehat{G^l_{\pi}}=\{\chi^i\;|\; 0 \leq i \leq n-1\}$,
$\widehat{G^r_{\pi}}=\{\eta^i\;|\; 0 \leq i \leq n-1\}$, and
$\widehat{G_{\pi}}=\{\chi^i\eta^j\;|\; 0 \leq i,j \leq n-1\},$
so there are group gradings
\begin{eqnarray}
H=\oplus_{\chi^i\in \widehat{G^l_{\pi}}} H^l_{i}
= \oplus_{\eta^j\in \widehat{G^r_{\pi}}} H^r_{j}
=\oplus_{0 \leq i,j \leq n-1} H_{ij},
\label{yyE6.2.1}
\end{eqnarray}
where $H_{ij}=H^l_{i}\cap H^r_{j}$. The first two of these are
strong gradings by Proposition \ref{yysec2.2}. Since we are assuming
\eqref{yyE6.1.1}, Proposition \ref{yysec5.2}(b) and Theorem
\ref{yysec5.2} confirm that
\begin{eqnarray}\label{yyE6.2extra} Z(H)
= H_0 = \widetilde{H_0} = H_{00}.\end{eqnarray} In particular, this
is a Hopf subalgebra of $H$, and hence is either $k[x]$ or $k[x^{\pm
1}],$ yielding the dichotomy - primitive or group-like - of
Definition \ref{yysec4.1}.

Note that the discussion in the proof of $(2) \Rightarrow (4)$
of Proposition \ref{yysec5.2}(b) applies: Each $H_{ij}$
is a torsion-free - and hence free - $H_0$-module of rank 1.
Fix a generator $u_{ij}$ of this free module, for
$i,j = 0, \ldots , n-1,$ with $u_{00}=1.$ Then part of
\eqref{yyE6.2.1} can be rewritten as
\begin{eqnarray}
H=\oplus_{0 \leq i,j \leq n-1} H_0u_{ij}
= \oplus_{0 \leq i,j \leq n-1} u_{ij}H_0.
\label{yyE6.2.2}
\end{eqnarray}
The definitions ensure that
$$\Xi^l_{\pi}(u_{ij}a)=\zeta^i u_{ij}a \quad \textit{and}
\quad \Xi^r_{\pi}(u_{ij}a)=\zeta^j u_{ij}a$$
for all $a\in H_0$.

Next we refine the notation for the elements of the groups
$K^l_{\pi}$ and $K^r_{\pi}$ of \eqref{yysec5.2}. Fix a nonzero
element $x\in H^l_1,$ and define $\kappa_l$ to be the automorphism
of $H^l_0$ defined by
$$\kappa_l(a)= xax^{-1}, \textit{ for all } a\in H^l_0. $$
In view of Proposition (\ref{yysec5.1})(b), for $i = 1, \ldots ,
n-1$ and for $j = 0, \ldots , n-1,$
\begin{eqnarray}
\kappa_l^i(a)=u_{ij}au_{ij}^{-1} \quad \textit{ for all } \quad a\in
H^l_0. \label{yyE6.2.3}
\end{eqnarray}
Similarly, choosing any non-zero element $t$ of $H^r_1$, $\kappa_r$
is the automorphism of $H^r_0$ defined by $\kappa_r(b)=tbt^{-1}
\textit{ for all } b \in H^r_0.$ Thus for $i = 0, \ldots , n-1$ and
for $j = 1, \ldots , n-1,$
\begin{eqnarray}
\kappa_r^j(b)=u_{ij}bu_{ij}^{-1} \quad \textit{ for all } \quad b
\in H^r_0. \label{yyE6.2.4}
\end{eqnarray}
Set $K^l_{\pi}= \langle \kappa_l^1 \rangle =
\{\kappa_l^i : 0 \leq i \leq n-1\}$ and
$K^r_{\pi} = \langle \kappa_r^1 \rangle =
\{\kappa_r^j : 0 \leq j \leq n-1\}$, subgroups of
$\mathrm{Aut}(H^l_0)$ and $\mathrm{Aut}(H^r_0)$ respectively.
By Proposition \ref{yysec5.2}(b) $K^l_{\pi}=\rho^l (G^r_{\pi})$
and $K^r_{\pi}= \rho^r (G^l_{\pi})$, (where here, of course, we
are carrying over the notation $\rho^l,\rho^r$ from
\S \ref{yysec5.2}). This implies that each $H_{0j}$
(and each $Q_{0j}$) is a $K^l_{\pi}$-eigenspace. Thus there are
$n$th roots of unity in $k,$ $\xi_l$ and $\xi_r$ say, such that
$$\kappa^1_l(u_{01})=\xi_l u_{01},
\quad \textit{and}\quad \kappa^1_r(u_{10})= \xi_r u_{10}.$$ Since the
actions of $K^l_{\pi}$ and $K^r_{\pi}$ on $H^l_0$ and $H^r_0$ are
faithful, by Proposition \ref{yysec5.1}(d), $\xi_l$ and $\xi_r$ are
primitive $n$th roots of unity. Using \eqref{yyE6.2.4} and
\eqref{yyE6.2.3},
$$u_{01}u_{10}=\kappa^1_r(u_{10})u_{01}=\xi_ru_{10}u_{01}
=\xi_r\kappa^1_l(u_{01})u_{10} = \xi_r \xi_l u_{01}u_{10}.$$
Thus $\xi_r=\xi_l^{-1}$. We set $\xi:=\xi_r.$ As $\xi$ and
$\zeta$ are primitive $n$th root of unity in $k,$ there are
integers $f$ and $f'$, coprime to $n$ and with $1 \leq
f,f' \leq n-1,$ such that $\xi=\zeta^f$ and $\zeta=\xi^{f'}$.

\begin{lemma} Retain all the above hypotheses and notation.
Let $i,i',j$ and $j'$ be integers from $\{0,1,\ldots , n-1 \}.$
Then $$u_{ij}u_{i'j'}=\xi^{i'j-ij'}u_{i'j'}u_{ij}.$$ In
particular,  $H_{ij}$ commutes with $H_{i'j'}$ if and only if
$i'j-ij'\equiv 0\mod n$.
\end{lemma}

\begin{proof} We can pass to the simple Artinian quotient ring
$Q(H)$ of $H$ by inverting $Z(H)\setminus \{0 \} = H_{00}
\setminus \{0\}.$ In this way we see that $Q(H)=\oplus Q_{ij}$ is
$\widehat{G_{\pi}}$-graded . Here, $\widehat{G_{\pi}} =
\widehat{G^l_{\pi}} \times \widehat{G^r_{\pi}}
\cong {\mathbb Z}_n \times {\mathbb Z}_n,$ and this a strong
grading, since each $Q_{ij}$ is a one-dimensional vector space
over the field $Q_0 = Q(Z(H)),$ spanned by the unit $u_{ij}$ of
$Q(H).$ Working inside $Q(H)$ and multiplying our original choices
of the elements $u_{ij}$ by non-zero central elements as
necessary, we may assume for the purposes of the present
calculation that $u_{ij}=u_{10}^iu_{01}^j$. Since the adjustments
are central, this will not affect the truth or otherwise of the
commutation relations we are trying to prove. Now
$$
\begin{aligned}
u_{ij}u_{i'j'}(u_{i'j'}u_{ij})^{-1}
&= (u_{10})^i(u_{01})^j(u_{10})^{i'}(u_{01})^{j'}
(u_{01})^{-j}(u_{10})^{-i}(u_{01})^{-j'}(u_{10})^{-i'}\\
&=\xi_r^{i'j-ij'}=\xi^{i'j-ij'}
\end{aligned}
$$
as required.
\end{proof}

\begin{proposition} Retain the above hypotheses and notation.
Let $i,j \in \{0,1,\ldots , n-1 \}$.
\begin{enumerate}
\item Conjugation of $Q(H)$ by $u_{ij}$ restricts to an
automorphism of $H,$ namely
$((\Xi^l_{\pi})^j(\Xi^r_{\pi})^{-i})^f$. That is,
$$((\Xi^l_{\pi})^j(\Xi^r_{\pi})^{-i})^f(x)=u_{ij}xu_{ij}^{-1}$$
for all $x\in H$. \item $(\Xi^l_{\pi})^f(x)=u_{01}xu_{01}^{-1}$
and $(\Xi^r_{\pi})^f=u_{10}^{-1}xu_{10}$ for all $x\in H$. \item
The element $u_{ij}$ is normal in $H$.
\end{enumerate}
\end{proposition}

\begin{proof} (a) By \eqref{yyE6.2.2}, $H=\oplus_{ij} u_{ij}H_0$
and the center $H_0$ is fixed by $\Xi_l$ and $\Xi_r$. It suffices
therefore to show that
$$((\Xi^l_{\pi})^j(\Xi^r_{\pi})^{-i})^f(u_{i'j'})
=u_{ij}u_{i'j'}u_{ij}^{-1}.$$
This follows from the above lemma.

(b,c) are special cases of (a).
\end{proof}

\subsection{The ideal $J_{iq}$}
\label{yysec6.3} Recall the definition in (\ref{yysec2.3}) of the
integral annihilator $J_{iq}$, and recall the comment which follows
Hypotheses \ref{yysec2.5}.

\begin{lemma}
Continue with the notations and hypotheses introduced so far
in \S 6. Let $i,j \in \{0,1,\ldots , n-1 \}$,
\begin{enumerate}
\item
$S(H^l_{i})=H^r_{-i}$ and $S(H_{ij})=H_{-j,-i}$,
(where the suffixes are interpreted $\mathrm{mod}\; n$).
\item
If $i\neq j$ then $\epsilon(H_{ij})=0$.
\item
$\epsilon(u_{ii})\neq 0$.
\item
If $i\neq j$ then $H_{ij}\subseteq J_{iq}$.
\end{enumerate}
\end{lemma}

\begin{proof} (a) By Proposition \ref{yysec2.1}(e),
$S(\Xi^l_{\pi})^{-1}=\Xi^r_{\pi}S,$ so that
$(\Xi^r_{\pi})^{-1}S=S\Xi^l_{\pi}$. Bearing in mind that $S$ is
injective, this implies that $S(H^l_{i})=H^r_{-i}$. The second
claim follows from the first, since $H_{ij}=H^l_i\cap H^r_j$.

(b) Let $i,j \in \{0,1, \ldots ,  n-1 \}$ and let $x\in H_{ij}$.
By Proposition \ref{yysec2.2}(a) and its right-hand version,
$\Delta(x)\in H^l_i\otimes H^r_j$. Therefore, using (a),
$$\epsilon(x)=m_H\circ(\mathrm{Id}\otimes S)\circ \Delta(x)
\in H^l_iH^l_{-j}=H^l_{i-j}.$$
Thus, if $i\neq j$, $\epsilon(x)\in H^l_0\cap H^l_{i-j}=\{0\}$.

(c) Suppose that $i=j.$ By (b) and Proposition
\ref{yysec2.2}(a),
\begin{eqnarray}
\quad \quad u_{ii}=m_H \circ (\epsilon\otimes \mathrm{Id})\circ
\Delta(u_{ii}) \in m_H \circ (\epsilon\otimes \mathrm{Id})
( H^l_i\otimes H^r_i) = \epsilon(H^l_{i})H^r_i.
\label{yyE6.3.1}
\end{eqnarray}
Since $H^l_i = \oplus_j H_{ij}$ it follows from \eqref{yyE6.3.1}
and (b) that $\epsilon(H_{ii})\neq 0$. But $H_{ii}=u_{ii}H_0$,
so $\epsilon(u_{ii})\neq 0$.

(d) By (b), $H_{ij} \subseteq \mathrm{ker} \epsilon$. But
$H_{ij}$ is $G^l_{\pi}-$invariant, so it follows that
$$ H_{ij} \subseteq \cap_{i \geq 0}(\Xi^l_{\pi})^i
(\mathrm{ker} \epsilon) = J_{iq}, $$ the equality being
Theorem \ref{yysec2.3}(e).
\end{proof}

\subsection{The twistor}
\label{yysec6.4}
We shall approach the coalgebra and algebra structure of the
prime Hopf algebra $H$ of GK-dimension one by first studying
a finite dimensional Hopf factor of $H,$ as follows. By
Theorem \ref{yysec5.2} $H_0$ is a central Hopf subalgebra of
$H$. Hence there is an exact sequence of Hopf algebras
$$0\to H_0\to H\to H_{tw}\to 0,$$
where $H_{tw}$ is the Hopf algebra $H/(\ker\epsilon\cap H_0)H$. We
call $H_{tw}$ the {\it twistor} of $H$. Recall that, by
\eqref{yyE6.2.2}, $H$ is a free $H_0-$module of rank $n^2$ on the
basis $\{u_{ij}\}.$ Clearly, $\dim_k H_{tw}=n^2$, with a basis
$\{v_{ij}\}_{0\leq i,j\leq n-1}$ where $v_{ij}$ is the image in
$H_{tw}$ of the element $u_{ij}$ for every pair $(i,j)$. It will be
convenient also to denote the ideal $(\ker\epsilon\cap H_0)H$ of $H$
by $J_{tw}.$ Note that, by Proposition \ref{yysec2.3}, $J_{tw}
\subseteq J_{iq}$.

\begin{examples}
(a) Let $H=H(n,t,\xi)$ be a Taft algebra as in \S \ref{yysec3.3}, so $io(H)
= n$. By \eqref{yyE3.3.1}, $im(H)=io(H)$ if and only if $t$ is
coprime to $n$. Replacing $g$ by $g^{t}$ and $\xi$ by $\xi' :=
\xi^t$, we may assume that $t=1$. The twistor $H(n,1,\xi)_{tw}$ of this
Hopf algebra is then
$$H(n,1,\xi')_{tw}:=k\langle g,x\rangle/(x^n=0,g^n=1, gx=\xi' xg),$$
with $\Delta(g)=g\otimes g$ and $\Delta(x)=x\otimes g+1\otimes x$.

(b) Let $B = B(n,w,\xi)$ be one of the examples from \S \ref{yysec3.4}.
As noted there, $B_0 = k[x^{\pm 1}]$, so that $B_{tw} = B/(x - 1)B$
and it follows at once that $B(n,w,\xi)_{tw} \cong H(n,1,\xi)_{tw}$.
\end{examples}

\subsection{Coalgebra structure of $H_{tw}$}
\label{yysec6.5}
The key to the coalgebra structure rests in the comultiplication
rules for the basis elements $v_{ij}$. Now $v_{00} = 1$ (because
$u_{00}=1$) and, thanks to Lemma \ref{yysec6.3}, after adjustment
by suitable non-zero scalars we can assume that, in $H$,
\begin{eqnarray}
\epsilon (u_{ij}) = \delta_{ij}, \label{yyE6.5.1}
\end{eqnarray}
so that, in $H_{tw}$,
\begin{eqnarray}
\epsilon (v_{ij}) = \delta_{ij}, \label{yyE6.5.2}
\end{eqnarray}
for all $i,j = 0,1, \ldots , n-1$.

\begin{lemma}
Continue with the notations and hypotheses of $\S 6.$ The
comultiplication $\Delta$ in $H_{tw}$ satisfies the following
properties.

(a) For $i = 0, \ldots , n-1,$
$$\Delta(v_{ii})=v_{ii}\otimes v_{ii}+\sum_{s\neq i}
c^{ii}_{ss}v_{is}\otimes v_{si},$$ with $c^{ii}_{ss} = c^{ss}_{ii}$ for all $i \neq s.$

(b) For $i,j = 0, \ldots , n-1,$ with $i \neq j$,
$$\Delta(v_{ij})=v_{ii}\otimes v_{ij}+v_{ij}\otimes v_{jj}
+\sum_{s\neq i,j} c^{ij}_{ss} v_{is}\otimes v_{sj}$$
for scalars $c^{ij}_{ss}$ satisfying the relations
$$c^{ij}_{tt}c^{it}_{ss}=c^{ij}_{ss}c^{sj}_{tt}$$
for all mutually distinct $i,j,s,t$ in $\{0, \ldots, n-1\}$.

(c) For all mutually distinct $i,s,t$ in $\{0, \ldots,
n-1\},$
$$c^{tt}_{ss}=c^{it}_{ss}c^{is}_{tt}=c^{ti}_{ss}c^{si}_{tt}.$$

\end{lemma}

\begin{proof} First, by Proposition \ref{yysec2.2}(a) and the
fact that $u_{ij} \in H_{ij} = H^l_i \cap H^r_j$, there are elements
$c^{ij}_{st}$ of $k$ such that for all $i,j = 0, \ldots , n-1$,
\begin{eqnarray}
\Delta (v_{ij}) \quad = \quad \sum_{s,t} c^{ij}_{st} v_{is}
\otimes v_{tj}. \label{yyE6.5.3}
\end{eqnarray}
Fix $i,$ consider (\ref{yyE6.5.3}) with $j=i,$ and apply
$m_H \circ (\mathrm{Id} \otimes \epsilon)$ and
$m_H \circ (\epsilon \otimes \mathrm{Id})$ to the right hand side.
Since the outcome in both cases is $v_{ii}$ we deduce from
(\ref{yyE6.5.2}) and the linear independence over $k$ of the
$v_{ij}$ that
\begin{eqnarray}
\Delta (v_{ii}) \quad = \quad v_{ii} \otimes v_{ii} +
\sum_{s,t \neq i} c^{ii}_{st} v_{is} \otimes v_{ti}.
\label{yyE6.5.4}
\end{eqnarray}
Similarly we have, for $i \neq j,$
\begin{eqnarray}
\Delta (v_{ij}) \quad = \quad v_{ii} \otimes v_{ij} + v_{ij}
\otimes v_{jj}+ \sum_{s\neq i,t\neq j} c^{ij}_{st} v_{is} \otimes v_{tj}.
\label{yyE6.5.5}
\end{eqnarray}
Applying $(\mathrm{Id} \otimes \Delta)$ and
$(\Delta \otimes \mathrm{Id})$ to \eqref{yyE6.5.4} shows that
$$
\begin{aligned}
(\mathrm{Id}\otimes & \Delta)\circ \Delta (v_{ii})
= v_{ii} \otimes v_{ii} \otimes v_{ii} +
\sum_{l,p \neq i}c^{ii}_{lp}(v_{ii}\otimes v_{il} \otimes v_{pi})+\\
&\quad \sum_{s,t \neq i} c^{ii}_{st} v_{is} \otimes(v_{tt}\otimes v_{ti}+
v_{ti}\otimes v_{ii})
+\sum_{l\neq t,p\neq i}\sum_{s,t \neq i}c^{ti}_{lp}c^{ii}_{st}(v_{is}
\otimes v_{tl}\otimes v_{pi}),
\end{aligned}
$$
and
$$
\begin{aligned}
(\Delta \otimes & \mathrm{Id})\circ \Delta (v_{ii})
= v_{ii} \otimes v_{ii} \otimes v_{ii} +
\sum_{w,u \neq i}c^{ii}_{wu}(v_{iw}\otimes v_{ui} \otimes v_{ii})+\\
&\quad
\sum_{s,t \neq i} c^{ii}_{st} (v_{is}\otimes v_{ss}+
v_{ii}\otimes v_{is})\otimes v_{ti}
+\sum_{w\neq i,u\neq s}\sum_{s,t \neq i}c^{is}_{wu}c^{ii}_{st}(v_{iw}
\otimes v_{us}\otimes v_{ti}),
\end{aligned}
$$
respectively. Using the above expressions and canceling out the
equal terms from the left-hand and right-hand sides of the equation
$$(\mathrm{Id}\otimes \Delta)\circ \Delta (v_{ii})
=(\Delta \otimes \mathrm{Id})\circ \Delta (v_{ii}) $$
we obtain
\begin{eqnarray}
\sum_{s,t \neq i} c^{ii}_{st} v_{is} \otimes v_{tt}\otimes v_{ti}
+\sum_{l\neq t,p\neq i}\sum_{s,t \neq i}c^{ti}_{lp}c^{ii}_{st}v_{is}
\otimes v_{tl}\otimes v_{pi}= \quad
\label{yyE6.5.6}
\end{eqnarray}
$$
\quad \sum_{s,t \neq i} c^{ii}_{st} v_{is}\otimes v_{ss}\otimes v_{ti}
+\sum_{w\neq i,u\neq s}\sum_{s,t \neq i}c^{is}_{wu}c^{ii}_{st}v_{iw}
\otimes v_{us}\otimes v_{ti}.
$$
Comparing the terms with $v_{is} \otimes v_{tt}\otimes v_{ti}$,
it follows that
\begin{eqnarray}
c^{ii}_{st}=0
\label{yyE6.5.7}
\end{eqnarray}
for all $s\neq t$. Thus we have proved the first part of (a).

Substituting \eqref{yyE6.5.7} into \eqref{yyE6.5.6} and re-arranging
the indices we have
$$
\sum_{l\neq w,p\neq i}\sum_{w\neq i}c^{wi}_{lp}c^{ii}_{ww}v_{iw}
\otimes v_{wl}\otimes v_{pi}= \sum_{w\neq i,u\neq p}\sum_{p\neq
i}c^{ip}_{wu}c^{ii}_{pp}v_{iw} \otimes v_{up}\otimes v_{pi}.
$$
This implies that
$$
\begin{cases}
c^{wi}_{lp}c^{ii}_{ww}=0 & p\neq l\\
c^{ip}_{wu}c^{ii}_{pp}=0 & w\neq u\\
c^{wi}_{pp}c^{ii}_{ww}=c^{ip}_{ww}c^{ii}_{pp}.&
\end{cases}
$$

In proving (b) and (c), let's look first at the case when $n=2$.
Since $v_{00}=1$, $c^{00}_{11}=0$. If $c^{11}_{00}$ is nonzero, then
$c^{01}_{wu}=c^{10}_{wu}=0$ for $w \neq u$ by the first two of the
above three equations, and so \eqref{yyE6.5.5} becomes
$$\Delta (v_{ij}) \quad = \quad v_{ii} \otimes v_{ij} + v_{ij}
\otimes v_{jj}$$ for $i \neq j.$ By comparing $(\mathrm{Id}\otimes
\Delta)\circ \Delta (v_{01})$ with $(\Delta\otimes \mathrm{Id})\circ
\Delta (v_{01})$ we obtain that $c^{11}_{00}=c^{00}_{11}=0$, a
contradiction. Hence in this case, thanks to (a), $v_{11}$ is a
group-like element. By \eqref{yyE6.5.5}, we have
$$\Delta (v_{01}) \quad = \quad v_{00} \otimes v_{01} + v_{01}
\otimes v_{11}+ c^{01}_{10} v_{01}\otimes v_{01}.$$ By comparing
$(\mathrm{Id}\otimes \Delta)\circ \Delta (v_{01})$ with
$(\Delta\otimes \mathrm{Id})\circ \Delta (v_{01})$ again we obtain
$c^{01}_{10}=0$. Similarly, $c^{10}_{01}=0$. Thus, in view of (a)
and these observations, we have proved that $v_{11}$ is a group-like
element and that both $v_{01}$ and $v_{10}$ are skew primitive
elements. Therefore (a), (b) and (c) hold in the case $n=2.$

We return now to the case where $n>2.$ Fix $i$ and $j$ with $i\neq
j$. We proceed as in the case $i=j,$ between \eqref{yyE6.5.5} and
\eqref{yyE6.5.7}; since the arguments are similar we give fewer
details. Calculating the left-hand and the right-hand sides of the
equation
$$(\mathrm{Id}\otimes \Delta)\circ \Delta (v_{ij})=
(\Delta\otimes \mathrm{Id})\circ \Delta (v_{ij}),$$ and canceling
equal terms (where we omit the detailed expressions), we find that
\begin{eqnarray}
\quad
\label{yyE6.5.8}
\end{eqnarray}
$$\sum_{s\neq j} c^{jj}_{ss} v_{ij}\otimes v_{js}\otimes v_{sj}
+\sum_{s\neq i,t\neq j} c^{ij}_{st}v_{is}
\otimes(v_{tt}\otimes v_{tj}+\sum_{w\neq t,u\neq j} c^{tj}_{wu}
v_{tw}\otimes v_{uj})\qquad $$
$$\qquad =
\sum_{s\neq i}c^{ii}_{ss}v_{is}\otimes v_{si}\otimes v_{ij}+
\sum_{s\neq i,t\neq j}c^{ij}_{st} (v_{is}\otimes v_{ss}+
\sum_{p\neq i,l\neq s} c^{is}_{pl} v_{ip}\otimes v_{ls})\otimes v_{tj}.
$$
By comparing the terms with $v_{is}\otimes v_{tt}\otimes v_{tj}$, we
see that $c_{st}^{ij}=0$ for all $s\neq t$. Hence the first claim in
(b) follows. Substituting this into \eqref{yyE6.5.8} and deleting
the equal terms on left and right, we have
\begin{eqnarray}
\quad
\label{yyE6.5.9}
\end{eqnarray}
$$\sum_{s\neq j} c^{jj}_{ss} v_{ij}\otimes v_{js}\otimes v_{sj}
+\sum_{s\neq i,j} c^{ij}_{ss}v_{is}
\otimes(\sum_{w\neq s,j} c^{sj}_{ww}
v_{sw}\otimes v_{wj})\qquad $$
$$\qquad =
\sum_{s\neq i} c^{ii}_{ss}v_{is}\otimes v_{si}\otimes v_{ij}+
\sum_{s\neq i, j}c^{ij}_{ss} (\sum_{p\neq i,s} c^{is}_{pp}
v_{ip}\otimes v_{ps}\otimes v_{sj}).
$$
Comparing the term $v_{ij}\otimes v_{ji}\otimes v_{ij}$, it follows
that
$$c^{ii}_{jj}=c^{jj}_{ii},$$
proving the rest of (a).

Comparing terms $v_{ij}\otimes v_{js}\otimes v_{sj}$ for $s\neq i,j$
yields
$$c^{jj}_{ss}=c^{ij}_{ss}c^{is}_{jj}.$$
Comparing terms $v_{is}\otimes v_{si}\otimes v_{ij}$ for $s\neq i,j$
we deduce
$$c^{ij}_{ss}c^{sj}_{ii}=c^{ii}_{ss}.$$
Combining the last two equations yields
\begin{eqnarray}
c^{tt}_{ss}=c^{it}_{ss}c^{is}_{tt}=c^{ti}_{ss}c^{si}_{tt}
\label{yyE6.5.10}
\end{eqnarray}
for all distinct $s,t,i$. That is, (c) is proved. Finally, by
comparing terms $v_{is}\otimes v_{st}\otimes v_{tj}$ for distinct
$i,j,s,t$, we obtain
$$c^{ij}_{ss}c^{sj}_{tt}=c^{ij}_{tt}c^{it}_{ss},$$
which is the second part of (b). This completes the proof.
\end{proof}

\subsection{Algebra and coalgebra structure of $H_{tw}$}
\label{yysec6.6}

\begin{proposition}
Keep all the notation introduced so far in $\S \ref{yysec6}$.
The following relations hold in $H_{tw}$.
\begin{enumerate}
\item
Let $g=v_{11}$. Then $g^n=1$ and for $i = 0,1, \ldots , n-1$,
$v_{ii}=g^i$.
\item
For all $i,j = 0,1, \ldots , n-1,$ with $i \neq j$, $v_{ij}^n=0$.
\item
For all $i,j = 0,1,\ldots , n-1$, $gv_{ij}=\xi^{i-j}v_{ij}g$.
In particular, $gv_{10}=\xi v_{10}g$ and $gv_{01}=\xi^{-1}v_{01}g$.
\item
The element $g$ is  group-like in $H_{tw}$. As a consequence
$$c^{tt}_{ss}=0\quad
{\text{and}}\quad
c^{it}_{ss}c^{is}_{tt}=c^{ti}_{ss}c^{si}_{tt}=0$$
for all distinct $s,t,i$.
\item
We may extend the expressions for $v_{ii}$ in (a) by defining
$v_{ij}:=g^i v_{0(j-i)}$ for all $i,j$, where $j-i$ is to be
interpreted modulo $n$ when $i>j.$ Once this is done,
$c^{ij}_{ss}=c^{(i-t)(j-t)}_{(s-t)(s-t)}$ for all $i,j,s,t$ with
$i,j$ and $s$ distinct.
\item
There exists $t$ coprime to $n$, $1 \leq t \leq n-1,$ such that
$k\langle v_{0t} \rangle$ contains $v_{0i}$ for all $i = 0,1,
\ldots , n-1$.
\item
As an algebra, $H_{tw}$ is isomorphic to the twistor
$H(n,1,\xi^t)_{tw}$ of the Taft algebra, when we take the
generators to be $y=v_{0t}$ and $g = v_{tt}.$
\end{enumerate}
\end{proposition}

\begin{proof}(a) Since $H_{tw}$ is graded, $g^n\in kv_{00}=k$. By
\eqref{yyE6.5.2}, $\epsilon(g^n)=1^n=1$. Hence $g^n=1$. In
particular, $g$ is invertible. Both $g^i$ and $v_{ii}$ are nonzero
elements in the $1$-dimensional space $kv_{ii}$. Then
\eqref{yyE6.5.2} shows that they are equal.

(b) Let $i \neq j.$ Since $H_{tw}$ is graded, $v_{ij}^n\in
kv_{00}=k$. The assertion follows from \eqref{yyE6.5.2}.

(c) This follows at once from Lemma \ref{yysec6.2}.

(d) By Lemma \ref{yysec6.5}(a)
$$\Delta(v_{11})=v_{11}\otimes v_{11}+\sum_{s\neq i}c^{11}_{ss}
v_{1s}\otimes v_{s1};$$ that is,
$$\Delta(g)=g\otimes g+\alpha$$
where $\alpha:=\sum_{s\neq i}c^{11}_{ss}v_{1s}\otimes v_{s1}$
commutes with $g\otimes g$ by (c). Since $g^n=1$,
$\Delta(g)^n=1\otimes 1$. This implies that
$$\sum_{i=0}^n {n \choose i}(g\otimes g)^{n-i}\alpha^{i}=1\otimes 1$$
which is equivalent to
$$\alpha
\big\{n(g^{n-1}\otimes g^{n-1})+\sum_{i=2}^n {n \choose i}
(g\otimes g)^{n-i}\alpha^{i-1}\big\}=0.$$
Since $\beta:=\sum_{i=2}^n {n \choose i}(g\otimes
g)^{n-i}\alpha^{i-1}$ is nilpotent by (b) and Lemma \ref{yysec6.2},
$n(g^{n-1}\otimes g^{n-1})+\beta$ is invertible. Hence $\alpha=0$,
and $g$ is group-like as claimed. Since $v_{ii}=v_{11}^i$ for $i=1,
\ldots , n-1,$ these elements too are group-like. Hence the
remaining claims in (d) follow from Lemma \ref{yysec6.5}(c).

(e) Since the proposed elements $v_{ij}$ are non-zero and lie in
the correct subspace, the first part is clear. The second part
follows by using Lemma \ref{yysec6.5}(b) and the fact that
$\Delta$ is an algebra homomorphism to expand the equation
$\Delta(v_{ij}) = \Delta ( g^t v_{(i-t)(j-t)})$.

(f) Set $B:=(H^l_0+J_{tw})/J_{tw}$. Then $B= \oplus_{0 \leq i\leq
n-1} kv_{0i}$. By (b) and Lemma \ref{yysec6.2}, $B$ is a local
ring with Jacobson radical
$$J =\oplus_{0<i \leq n-1}kv_{0i} = ( H^l_0 \cap J_{iq})+J_{tw}/
 J_{tw}.$$
However, $B$ is a (finite dimensional) factor of $H^l_0$, which is a
Dedekind domain by Hypotheses \ref{yysec2.5}(c). Thus $B$ is a
principal ideal ring, and so its radical $J$ is generated by the
(unique) element $v_{0t} \in J\setminus J^2$. By induction,
$J^i=k(v_{0t})^i +J^{i+1}$ provided $J^i\neq 0$. Thus the
$k$-subalgebra of $B$ generated by $v_{0t}$ includes all $v_{0i}$
and therefore equals $B.$ That is, $B=k[y]/(y^n)$ where $y:=v_{0t}$.
It is clear from the grading that $t$ must be coprime to $n$.

(g) By (e) and (f), $H_{tw}$ is generated by $g$ and $y$. Our
calculations so far show that, as an algebra, $H_{tw}$ is a factor
of the algebra
$$k\langle g,y\rangle/\langle y^n=0,g^n=1, gy=\xi^{-t}yg
\rangle.$$ But this latter algebra is isomorphic to
$H(n,1,\xi^t)_{tw}$, and so comparing dimensions we see that in fact
the two algebras are isomorphic.
\end{proof}

Retain all the notation and hypotheses of \S \ref{yysec6}.
We can now completely describe the twistor of $H$:

\begin{theorem} Suppose $(H,\pi)$ satisfies Hypotheses \ref{yysec2.5} and
\eqref{yyE6.1.1}. There exists an $n$th primitive root of 1, $\xi$,
such that the twistor $H_{tw}$ of $H$ is
isomorphic as a Hopf algebra to the twistor $H(n,1,\xi)_{tw}$ of the
Taft algebra $H(n,1,\xi)$, as described in Examples \ref{yysec6.4}.
The basis elements $\{v_{ij} : 0 \leq i,j \leq n-1 \}$ can be chosen
to be
$$v_{ij} := \begin{cases}
g^iy^{j-i} & i \leq j,\\ g^iy^{n+j-i}, &i>j.
\end{cases}$$
\end{theorem}

\begin{proof} Let $t$ be the integer determined by
Proposition \ref{yysec6.6}(f) and let $t'$ be an integer
such that $t'$ is the multiplicative inverse of $t$ in $C_n$.
We replace $\pi$ by its $t'$th power $\pi^{* t'}$
in $G(H^{\circ})$. It is clear that Hypotheses
\ref{yysec2.5} and \eqref{yyE6.1.1} hold for this new $\pi$. Doing
so allows us to assume that $y=v_{01}$. Adjusting $g$ accordingly,
we continue to take $g=v_{11}$. By Proposition \ref{yysec6.6}(g),
$H_{tw}\cong H(n,1,\xi)_{tw}$ as algebras. Since $\epsilon$ is
already determined by \eqref{yyE6.5.2}, and the antipode, being
the inverse of the identity map under convolution, is determined
by the bialgebra structure, it remains only to compare the
coalgebra structures. By Proposition \ref{yysec6.6}(d)
$\Delta(g)=g\otimes g,$ and Lemma \ref{yysec6.5}(b) shows that
\begin{eqnarray}
\label{yyE6.6.1}\quad \quad \qquad \Delta(y)=1\otimes y+y\otimes
g+\sum_{s\neq 0,1}c^{01}_{ss} v_{0s}\otimes v_{s1} =1\otimes
y+y\otimes g+(y^2\otimes y^2)f
\end{eqnarray}
where $f\in H_{tw}\otimes H_{tw}$. The result will follow if we can
show that \begin{eqnarray} c^{01}_{ss}=0 \textit{ for all } s>1.
\label{yyEextra}\end{eqnarray} Using \eqref{yyE6.6.1} and the facts
that $y^n=0$ and $yg=\xi^t gy$, we calculate that, for all $1<j \leq
n-1$,
\begin{eqnarray}
\label{yyE6.6.2} \quad \Delta(y^{j})=1\otimes y^{j}+y^{j}\otimes
g^{j}+ (\sum_{i=1}^{j-1} {j \choose i}_{\xi^t}y^i\otimes g^i
y^{j-i}) +h_j,\end{eqnarray}
 where the total $y$-degree of $h_j$ is greater than
$j+1$. Since $\xi^t$ is a primitive $n$th root of unity, the
coefficient
 \begin{eqnarray}\label{yyE6.6.3}c^{0j}_{ii}=
{j \choose i}_{\xi^t}\neq 0 \end{eqnarray}
for all $1 \leq i \leq j-1.$ It follows from the second equation in
Proposition \ref{yysec6.6}(d) that
$$c^{0i}_{jj}=0$$
for all $0<i<j$, proving \eqref{yyEextra}. It is now clear that
the final adjustments to the choice of basis $\{v_{ij}\}$ can be
made as proposed.
\end{proof}

\subsection{Lifting the coalgebra structure}
\label{yysec6.7} In the previous two subsections we have shown
that $v_{11}$ is group-like and $v_{01}$ is
$(v_{11},1)$-primitive. The goal of this section is to show that
$u_{11}$ is group-like and $u_{01}$ is $(u_{11},1)$-primitive.

Recall from \ref{yyE6.2extra} that $Z(H)=H_0$ is either $k[x]$ or
$k[x^{\pm 1}]$. We shall use $z$ to denote $x$ in the first case and
$x-1$ in the second case, so that $$J_{tw} \cap H_0 = \ker \epsilon
\cap H_0 = \langle z \rangle.$$  The following is clear:

\begin{lemma}
Suppose $(H,\pi)$ satisfies Hypotheses \ref{yysec2.5} and
\eqref{yyE6.1.1}.
\begin{enumerate}
\item
$H/z^nH$ is finite dimensional for all $n$, and
$\bigcap_{n\geq 0} z^n H=\{0\}$.
\item
$H/zH=H_{tw}$.
\item
$H_{tw}/Jac(H_{tw})=H_{iq}=k{\mathbb Z}_n$ where $Jac(-)$
is the Jacobson radical.
\end{enumerate}
\end{lemma}

By Proposition \ref{yysec2.2}(a) and the fact that $H_0u_{ij} =
H_{ij} = H^l_i \cap H^r_j,$ there are elements $C^{ij}_{st}$ of
$H_0\otimes H_0$ such that for all $i,j = 0, \ldots , n-1,$
\begin{eqnarray}
\Delta (u_{ij}) \quad = \quad \sum_{s,t} C^{ij}_{st} u_{is}
\otimes u_{tj}.
\label{yyE6.7.1}
\end{eqnarray}
Clearly, $(\epsilon\otimes \epsilon)(C^{ij}_{st})=c^{ij}_{st}$. For
$C^{ij}_{st}=\sum c_i \otimes d_i\in H_0\otimes H_0$, let's write
$[C^{ij}_{st}]_{13}$ to denote the element $\sum_i c_i\otimes
1\otimes d_i\in H_0\otimes H_0\otimes H_0$.

\begin{proposition}
Suppose $(H,\pi)$ satisfies Hypotheses \ref{yysec2.5} and
\eqref{yyE6.1.1}.
\begin{enumerate}
\item The element $u_{11}$ of $H$ is group-like. That is,
$C^{11}_{11}=1$ and $C^{11}_{st}=0$ for all $(s,t)\neq (1,1)$.
Furthermore, we may define $u_{ii}:=u_{11}^i$ for all $i=1,\cdots,
n-1$, so that $u_{ii}$ is group-like for all $i=1,\ldots, n-1$.
\item After a possible adjustment, $u_{01}$ is a
$(u_{11},1)$-primitive; that is, $$\Delta(u_{01}) =1\otimes
u_{01}+u_{01}\otimes u_{11}.$$
\end{enumerate}
\end{proposition}

\begin{proof} Let $C^{ij}_{st}$ be defined as in
\eqref{yyE6.7.1}. We will repeat some of the computations in
(\ref{yysec6.5}) and (\ref{yysec6.6}) to prove the following
preliminary steps towards (a) and (b):
\begin{eqnarray}
[(\Delta \otimes 1) C^{ij}_{st}][C^{is}_{\alpha\beta}\otimes 1]= [(1
\otimes \Delta) C^{ij}_{\alpha\beta}] [1\otimes C^{\beta j}_{st}]
\label{yyE6.7.2}
\end{eqnarray}
for all $i,j,s,t,\alpha,\beta;$
\begin{eqnarray}
\label{yyE6.7.3} (\epsilon\otimes 1) C^{ij}_{ii}=1\quad
{\text{and}}\quad (1\otimes \epsilon) C^{ij}_{jj}=1
\end{eqnarray}
for all $i$ and $j;$
\begin{eqnarray}
\label{yyE6.7.4}
(\epsilon\otimes 1) C^{ij}_{st}=0\quad
{\text{and}}\quad
(1\otimes \epsilon) C^{ij}_{st}=0
\end{eqnarray}
for all $i,j,s,t$ with $s\neq t;$
\begin{eqnarray}
\label{yyE6.7.5}
C^{ij}_{st}=0
\end{eqnarray}
for all $i,j,s,t$ with $s\neq t;$ \begin{eqnarray}\label{yyE6.7.5a}
C^{11}_{ss}=0 \textit{ for all } s\neq 1, \textit{ so }
\Delta(u_{11}) =C^{11}_{11} u_{11}\otimes u_{11}.
\end{eqnarray}

\emph{Proof of} \eqref{yyE6.7.2}: Applying $(\Delta\otimes 1)$ and
$(1\otimes \Delta)$ to \eqref{yyE6.7.1} we have
$$(\Delta\otimes 1)\Delta(u_{ij})=
\sum[(\Delta\otimes 1) C^{ij}_{st}][C^{is}_{\alpha \beta}\otimes 1]
u_{i\alpha}\otimes u_{\beta s}\otimes u_{tj};$$
$$(1\otimes \Delta)\Delta(u_{ij})=
\sum[(1\otimes \Delta)C^{ij}_{\alpha\beta}][1\otimes C^{\beta
j}_{st}] u_{i\alpha}\otimes u_{\beta s}\otimes u_{tj}.$$ Now
\eqref{yyE6.7.2} follows from the above equations and the fact that
$(\Delta\otimes 1)\circ \Delta=(1\otimes \Delta)\circ \Delta$.

\emph{Proof of} \eqref{yyE6.7.3}: By definition and Lemma
\ref{yysec6.5}(a,b), $(\epsilon\otimes \epsilon)
(C^{ij}_{ii})=c^{ij}_{ii}=1$. Hence $w:= (\epsilon\otimes
1)(C^{ij}_{ii})$ is nonzero. Applying $\epsilon\otimes \epsilon
\otimes 1$ to \eqref{yyE6.7.2} with $\alpha=\beta=s=t=i$ and using
the fact that $(\epsilon\otimes 1)\circ \Delta(c) =1\otimes c$, we
obtain
$$[(1\otimes \epsilon\otimes 1)(1\otimes C^{ij}_{ii})]
[((\epsilon\otimes \epsilon)C^{ii}_{ii})\otimes 1] =[(\epsilon
\otimes 1\otimes 1)[C^{ij}_{ii}]_{13}] [(\epsilon \otimes \epsilon
\otimes 1)(1\otimes C^{ij}_{ii})].$$ Since $(\epsilon\otimes
\epsilon)C^{ii}_{ii}=c^{ii}_{ii}=1$, the above equation reduces to
$$1\otimes w=(1\otimes w)(1\otimes w),$$
and since $H_0\otimes H_0\otimes H_0$ is a domain, $w=1$. The second
equation $(1\otimes \epsilon) C^{ij}_{jj}=1$ follows by symmetry.

\emph{Proof of} \eqref{yyE6.7.4}: Applying $1\otimes \epsilon\otimes
\epsilon$ to \eqref{yyE6.7.2} with $\alpha=\beta=s\neq t$, we find
that
$$[(1\otimes 1\otimes \epsilon)([C^{ij}_{st}]_{13})]
[(1\otimes \epsilon)(C^{is}_{ss})\otimes 1] =[(1\otimes
\epsilon\otimes 1)(C^{ij}_{ss}\otimes 1)] [1\otimes (\epsilon\otimes
\epsilon)(C^{sj}_{st})].
$$
By \eqref{yyE6.7.3} $(1\otimes \epsilon)(C^{is}_{ss})=1$ and by
Lemma \ref{yysec6.5}(a,b), $(\epsilon\otimes \epsilon)(C^{sj}_{st})
=c^{sj}_{st}=0,$ since $s\neq t$. Therefore $(1\otimes 1\otimes
\epsilon)([C^{ij}_{st}]_{13})=0;$ that is, $(1\otimes
\epsilon)(C^{ij}_{st})=0$ for all $i,j,s,t$ with $s\neq t$. This is
the second equation in \eqref{yyE6.7.4}. The first one can be
deduced similarly.

\emph{Proof of} \eqref{yyE6.7.5}: Applying $1\otimes \epsilon\otimes
1$ to \eqref{yyE6.7.2} with $\alpha\neq \beta=s=t$, yields
$$[C^{ij}_{ss}]_{13}[(1\otimes \epsilon)(C^{is}_{\alpha s})\otimes 1]
=[C^{ij}_{\alpha s}]_{13}(1\otimes (\epsilon\otimes
1)(C^{sj}_{ss})).
$$
Since $(1\otimes \epsilon)(C^{is}_{\alpha s})=0$ by
\eqref{yyE6.7.4} and $(\epsilon\otimes 1)(C^{sj}_{ss})=1$ by
\eqref{yyE6.7.3}, we have $[C^{ij}_{\alpha s}]_{13}=0$. This is
equivalent to $C^{ij}_{\alpha s}=0$.

\emph{Proof of} \eqref{yyE6.7.5a}: By \eqref{yyE6.7.5}
$\Delta(u_{11})=\sum_s C^{11}_{ss} u_{1s}\otimes u_{s1}$. Let
$p_s=C^{11}_{ss} u_{1s}\otimes u_{s1}$ and $p=\sum_s p_s$. Thus
each $p_s$ is homogeneous in $H_{1s}\otimes H_{s1}$ and so $p_s$
commutes with $p_{s'}$ for all $s,s'$ by Lemma \ref{yysec6.2}.
Suppose that $p_s\neq 0$ for some $s\neq 1$, and fix $m \geq 1$
such that the image of $p_s$ is non-zero in $B:= H/(z^m)\otimes
H/(z^m),$  using Lemma \ref{yysec6.7}(a). To keep the notation
within reasonable bounds we shall use the same symbol in what
follows for the images of elements in $B$, as for the elements
themselves in $H \otimes H.$ Since the image of $C^{11}_{11}$ in
$H/\langle z \rangle \otimes H/\langle z \rangle = H_{tw} \otimes
H_{tw}$ is $1 \otimes 1$ by Proposition (\ref{yysec6.6})(d),
$C^{11}_{11}$ is invertible in $B$. On the other hand, for all $s
\neq 1,$ the image of $C^{11}_{ss}$ in $H_{tw} \otimes H_{tw}$ is
0 by Proposition \ref{yysec6.6}(d), so the image in $B$ of
$C^{11}_{ss}$ is in the Jacobson radical $J:=Jac(B)$ for all $s
> 1.$ Let $w$ be the maximum integer such that $C^{11}_{ss}\in
J^w$ for all $s\neq 1$. Fix $i,$ $ 1 \leq i \leq n-1.$ Because
$u_{11}^n\in H_0$, $p^n=\Delta(u_{11}^n)\in H_0\otimes H_0$.
Considering $H\otimes H$ as a $({\mathbb Z}_n)^{\times 4}$-graded
algebra, the degree $(0,i,i,0)$ part of $p^n$ is
$$n p_1^{n-1} p_{i+1}+f_i(p_s)$$
where $f_i$ is a polynomial in $\{p_s : 0 \leq s \leq n-1\}$ whose
total degree in $p_0, p_2, \ldots,p_{n-1}$ is at least two. But
$p^n\in H_0\otimes H_0$, and so $p^n$ is homogeneous of degree
$(0,0,0,0).$ It follows that $n p_1^{n-1} p_{i+1}+f_i(p_s)=0$ for
all $i = 1, \ldots , n-1.$ Therefore
$$np_1^{n-1}p_{i+1}=-f_i\in J^{2w};$$
that is,
$$p_{i+1}=(np_1^{n-1})^{-1}(-f_i)\in J^{2w},$$
for all $i\neq 0$. This is a contradiction to the definition of $w$.
Therefore $p_s=0$ for all $s\neq 0$. In other words, $C^{11}_{ss}=0$
for all $s\neq 1$.

Now we are going to prove (a) and (b).

(a) Recall that $\epsilon(u_{11})=1$ and $H_0$ is the center of $H$.
By \eqref{yyE6.7.5} and \eqref{yyE6.7.5a},
$$1=\epsilon(u_{11})=m_H\circ(S\otimes 1)\circ \Delta(u_{11})=
m_H[(S\otimes 1)(C^{11}_{11})(S(u_{11})\otimes u_{11})]
=fS(u_{11})u_{11}$$ where $f=m_H\circ(S\otimes 1)(C^{11}_{11})$.
Hence $u_{11}$ is invertible.

Since $u_{11}$ is invertible, so is $\Delta(u_{11})$, and thus
$C^{11}_{11}$ is invertible. Since $C^{11}_{11}\in H_0\otimes H_0$,
and $H_0=k[x]$ or $H_0=k[x^{\pm 1}]$, there is a non-zero scalar $a$
such that $C^{11}_{11}=a$ in the first case or $a x^i\otimes x^j$ in
the second case. It follows from
$$(\epsilon\otimes 1)\Delta(u_{11})=u_{11}=
(1\otimes \epsilon)\Delta(u_{11})$$ that $a=1$ in both cases and
$i=j=0$ in the second case. Thus $C^{11}_{11}=1$. The rest of (a) is
clear.

(b) We claim first that \begin{eqnarray} \label{primit}
\Delta(u_{01})=C^{01}_{00}u_{00}\otimes
u_{01}+C^{01}_{11}u_{01}\otimes u_{11}. \end{eqnarray} If $n=2$,
\eqref{primit} follows from \eqref{yyE6.7.5}. Suppose now that
$n>2,$ and let $i\neq j$. Set $s=t\neq j$ and $\alpha=\beta=j$ in
\eqref{yyE6.7.2} to see that
$$[(\Delta \otimes 1) C^{ij}_{ss}][C^{is}_{jj}\otimes 1]=
[(1 \otimes \Delta) C^{ij}_{jj}] [1\otimes C^{jj}_{ss}],$$ and
this is zero because $C^{jj}_{ss}=0$ by (a). Since $H_0\otimes
H_0\otimes H_0$ is a domain, either $C^{ij}_{ss}=0$ or
$C^{is}_{jj}=0$. However, by Proposition \ref{yysec6.6}(e) and
\eqref{yyE6.6.3}, $C^{ij}_{ss}\neq 0$ for all $i<s<j$. Therefore
$C^{is}_{jj}=0$ for all $i<s<j$. In particular, $C^{01}_{jj}=0$
for all $j>1$, so that \eqref{primit} is proved.

Next we claim that
\begin{eqnarray}\label{finalpush} C^{01}_{00}=C^{01}_{11}=
1\otimes 1.\end{eqnarray}
Comparing $(\Delta\otimes 1)\Delta(u_{01})$ with $(1\otimes
\Delta)\Delta(u_{01})$ we obtain the following three equations:
\begin{align}
(\Delta\otimes 1)C^{01}_{00}&=[(1\otimes \Delta)C^{01}_{00}]
[1\otimes C^{01}_{00}];\label{yyE6.7.6}\\
[(\Delta\otimes 1)C^{01}_{11}][C^{01}_{00}\otimes 1]&= [(1\otimes
\Delta)C^{01}_{00}][1\otimes C^{01}_{11}];
\label{yyE6.7.7}\\
[(\Delta\otimes 1)C^{01}_{11}][C^{01}_{11}\otimes 1]&= (1\otimes
\Delta)C^{01}_{11}.\label{yyE6.7.8}
\end{align}

Recall once again the dichotomy of \eqref{yyE6.2.2}: $H_0$ is a
central subHopf algebra, either $H_0 = k[x]$ or $H_0 = k[x^{\pm
1}].$
 For $f\in H_0\otimes H_0$ or $f\in
H_0\otimes H_0\otimes H_0$, write $d(f)$ and $\partial(f)$ for the
degree of $f$ with respect to $x\otimes 1$ and $1\otimes x$
respectively in the first case, and with respect to $x\otimes
1\otimes 1$ and $1\otimes 1 \otimes  x$ respectively in the second
case. Write $C^{01}_{00}=\sum_{i,j}a_{ij} x^i\otimes x^j$ where
$a_{00}=c^{01}_{11}=1$.

\noindent \emph{Case 1}: \emph{Primitive case}: Suppose that
$H_0=k[x],$ with $x$ primitive.  Now \eqref{yyE6.7.6} becomes
$$\sum a_{ij}(x\otimes 1+1\otimes x)^i\otimes x^j=
[\sum a_{ij} x^i\otimes (x\otimes 1+1\otimes x)^j] [\sum a_{ij}
1\otimes x^i\otimes x^j].$$ It follows that
$\partial(C^{01}_{00})=\partial(C^{01}_{00})+\partial(C^{01}_{00}),$
so that $\partial(C^{01}_{00})=0$ and $C^{01}_{00}=\sum_i
a_ix^i\otimes 1$. In this case, writing $h:=\sum_i a_i x^i$,
\eqref{yyE6.7.6} simplifies to
$$\Delta(h)=(h\otimes 1)(1\otimes h)=h\otimes h.$$
But the only group-like element in $k[x]$ is 1, so that
$C^{01}_{00}=1$. A similar argument shows that $C^{01}_{11}=1$, and
so \eqref{finalpush} is true in the case $H_0=k[x]$.

\noindent \emph{Case 2}: \emph{Group-like case}: Suppose that
$H_0=k[x^{\pm 1}],$ with $x$ group-like. Since we can switch $x$
and $x^{-1}$ we may assume that $\partial(C^{01}_{00})\geq 0$. In
this case \eqref{yyE6.7.6} becomes
$$\sum a_{ij}(x\otimes x)^i\otimes x^j=
[\sum a_{ij} x^i\otimes (x\otimes x)^j] [\sum a_{ij} 1\otimes
x^i\otimes x^j].$$ Again $\partial(f)=2\partial(f)$ and hence
$\partial(f)=0$. This says that there is no positive power of $x$
in the second component. Replacing $x$ with $x^{-1}$ we deduce
that $f=\sum_i a_i x^i\otimes 1$ and, as in the primitive case,
$h:=\sum_i a_i x^i$ is group-like. Therefore $h=x^i$ and
$C^{01}_{00}=x^i\otimes 1$. Replacing $u_{01}$ by $u_{01}x^{-i}$
we have $C^{01}_{00}=1$. A similar argument involving
\eqref{yyE6.7.8} implies that $C^{01}_{11}=x^s\otimes 1$. Then
\eqref{yyE6.7.7} forces $s=0$ and the proof of \eqref{finalpush}
is complete in this case also.
\end{proof}

\begin{theorem}
Suppose $(H,\pi)$ satisfies Hypotheses \ref{yysec2.5} and
\eqref{yyE6.1.1}. Retain all the notations introduced so far in \S 6
for $H$ ; so, in particular, $H$ is a free $H_0-$module with basis
of $G^l_{\pi} \times G^r_{\pi}-$eigenvectors $\{u_{ij} : 0 \leq i,j
\leq n-1 \},$ and after Proposition \ref{yysec6.7} $u_{00} = 1, \;
u_{ii} = u_{11}^i$ is group-like for $i = 1,\dots , n-1,$ and
$u_{01}$ is $(u_{11},1)-$primitive. Then, after multiplying as
necessary by units of $H_0,$ we can assume that $u_{ij}=u_{11}^i
u_{01}^{j-i}$ for all $i<j,$ and $u_{ij} = u_{11}^iu_{01}^{n+j-i}$
for $i >j.$ In particular, therefore,
$$\Delta(u_{0j})=\sum_{s} {j \choose s}_{\xi^t}
u_{0s}\otimes u_{sj}$$ and $C^{ij}_{ss}=c^{ij}_{ss}1\otimes 1$ for
all $i,j,s$.
\end{theorem}

\begin{proof} Define $\bar{u}_{ij}:=u_{11}^iu_{01}^{j-i}$
for $i\leq j$ and $\bar{u}_{ij}:=u_{11}^iu_{01}^{n+j-i}$ for $i>j$.
Thus, by Proposition \ref{yysec6.7},
\begin{eqnarray}\label{del1}\Delta(\bar{u}_{0j})=\sum_{s} {j \choose s}_{\xi^t}
\bar{u}_{0s}\otimes \bar{u}_{sj}.\end{eqnarray}

We show first that there are units $f_j \in H_0$ such that
$\bar{u}_{0j}=u_{0j}f_j$ for $j=0, \ldots, n-1.$ Since
$\bar{u}_{0j}\in H_{0j}$ there are certainly $f_j\in H_0$ such
that these equations hold. Then
\begin{eqnarray}\label{del2}\Delta(\bar{u}_{0j})=
\Delta(u_{0j}f_j)=(\sum_{s} C^{0j}_{ss}u_{0s}\otimes u_{sj})
\Delta(f_j).\end{eqnarray} As in the proof of \eqref{yyE6.7.5a},
these expressions are $\mathbb{Z}_n^4-$graded. Equating the degree
$(0,0,0,j)$ parts of the two expressions \eqref{del1} and
\eqref{del2} for $\Delta(\bar{u}_{0j}),$
\begin{eqnarray}\label{del3}1\otimes f_j=C^{0j}_{00}\Delta(f_j) \in H_0
\otimes H_0.\end{eqnarray} Comparing degrees $\partial(-)$ in $1
\otimes x$ in \eqref{del3} gives
$$ \partial(1 \otimes f_j) = \partial(C^{0j}_{00}) +
\partial(1 \otimes f_j), $$
since $\partial (1 \otimes f_j) = \partial(\Delta (f_j))$ whether
$H_0$ is $k[x]$ or $k[x^{\pm 1}].$ Therefore $C^{0j}_{00} = h
\otimes 1$ for some $h \in H_0$, which is non-zero since
$(\epsilon \otimes \epsilon)(C^{0j}_{00}) = 1 \otimes 1$ by
\eqref{yyE6.6.2}. Substituting this in \eqref{del3} and applying
$1 \otimes \epsilon,$ we deduce that $\epsilon (f_j) \neq 0$ and
then that $f_j$ is a unit of $H_0$.

Thus $f_j$ is either $\lambda_j$ or $\lambda_j x^{l_j},$ where
$\lambda_j \in k \setminus \{0\}$ and $l_j \in \mathbb{Z},$
depending on whether $H_0$ is $k[x]$ or $k[x^{\pm 1}].$ For each
$j = 0 , \ldots , n-1,$ consider the image of the equation $$
\bar{u}_{0j} \; = \; u_{01}^j \; = \; u_{0j}f_j $$ in $H_{tw},$
where it yields
$$ v_{01}^j \; = \; v_{0j}\varepsilon (f_j) \; = \; v_{01}^j
\varepsilon(f_j), $$ in view of the final statement in Theorem
(\ref{yysec6.6}). Hence
$$ \varepsilon (f_j) \; = \; 1, $$
so that $\lambda_j = 1.$ Second, in the case where $H_0 = k[x^{\pm
1}]$, we can now replace $u_{0j} = u_{01}^j x^{-l_j}$ by
$u_{0j}x^{l_j} = u_{01}^j$, as required.

Now let $j \geq i \geq 1,$ and consider $u_{11}u_{01}^{j-i};$
since $H_{ij} = u_{ij}H_0,$ we can write
$$ u_{11}^i u_{01}^{j-i} \; = \; u_{ij}f_{ij} $$ for some $f_{ij}
\in H_0.$ Thus $$ H_{0(j-i)} \; = \; u_{01}^{j-i}H_0 \; = \;
u_{11}^{-i}u_{ij}f_{ij}H_0, $$ so that $f_{ij}$ must be a unit of
$H_0$. Therefore $u_{11}^i u_{01}^{j-i}H_0 = H_{ij}, $ so we can
replace our initial choice of $u_{ij}$ with $$u_{ij} \; := \;
u_{11}^i u_{01}^{j-i}, $$ as claimed. A similar calculation
applies when $i > j.$
\end{proof}

\subsection{Primitive case}
\label{yysec6.8}
\begin{proposition} Assume $(H,\pi)$ satisfies Hypotheses \ref{yysec2.5} and \eqref{yyE6.1.1}. Suppose that $H$ is primitive -
that is, the Hopf subalgebra $H_0 = Z(H)$ is a polynomial algebra
$k[x].$ Then $H$ is isomorphic as a Hopf algebra to the Taft algebra
$H(n,t,\xi)$ of Example \ref{yysec3.3} with $\gcd(t,n)=1$.
\end{proposition}

\begin{proof} By Proposition \ref{yysec6.7}, $g:=u_{11}$
is group-like and $y:=u_{01}$ is $(g,1)$-primitive. Since the only
group-like in $H_0$ is 1, $g^n=1$. By Lemma (\ref{yysec6.2}) and the
quantum binomial theorem, $y^n$ is primitive. Up to a nonzero scalar
the only primitive element in $H_0$ is $x$, so we may assume that
$y^n=x$. It then follows that there is a Hopf algebra homomorphism
$\phi: H(n,t,\xi)\to H$. By Theorem \ref{yysec6.7}, $H$ is generated by
$g$ and $y$, both in the image of $\phi$, and hence $\phi$ is
surjective. Since both $H(n,t,\xi)$ and $H$ are prime of GK-dimension 1,
$\phi$ is an isomorphism.
\end{proof}

\subsection{Group-like case}
\label{yysec6.9}

\begin{proposition}
Assume that $(H,\pi)$ satisfies Hypotheses \ref{yysec2.5} and
\eqref{yyE6.1.1}. Suppose that $H$ is group-like - that is, the Hopf
subalgebra $H_0 = Z(H)$ is a Laurent polynomial algebra $k[x^{\pm
1}].$ Then, for appropriate $\xi$ and $w$, $H$ is isomorphic as a
Hopf algebra to one of the algebras $B(n,w,\xi)$ of Example
\ref{yysec3.4}.
\end{proposition}

\begin{proof} By Proposition \ref{yysec6.7}, $g:=u_{11}$
is group-like and $y:=u_{01}$ is $(g,1)$-primitive, and these
elements together with $x$ generate $H$ as a $k-$algebra by Theorem
\ref{yysec6.7}. By the $C_n \times C_n -$grading of $H$, $g^n$ is a
group-like element in $H_0= H_{00},$ and so there is a $w\geq 0$
such that $g^n=x^w ,$ (noting that we can replace $x$ by $x^{-1}$
if $w$ is negative).

Suppose for a contradiction that $w=0$. Then $g^n=1$ and hence
$y^n$ is a primitive element as in Proposition \ref{yysec6.8}. But
$H_0$ has no primitive elements, so we have a contradiction.
Therefore we can assume that $w>0$.

By Lemma \ref{yysec6.2}, $yg=\zeta^a gy$ where $\gcd(a,n)=1$. This
implies that $y^n$ is a $(g^n,1)$-primitive element in $H_0.$
Therefore, after dividing if necessary by a non-zero scalar,
$$y^n=1-g^n.$$ Therefore, in view of Theorem \ref{yysec3.4}(c),
there is a Hopf algebra homomorphism
$\phi:B(n,\xi,w,i_0)\to H$ given by $\phi: x\to x, y\to y, g\to g$
where $\xi=\zeta^a$.
Since $H$ is generated by $x^{\pm 1}, g, y$, $\phi$ is surjective.
However, both $B$ and $H$ are prime of GK-dimension 1, so $\phi$ is
an isomorphism.
\end{proof}

The Main theorem \ref{yysec0.5} follows from Theorem \ref{yysec4.1}
and Propositions \ref{yysec6.8} and \ref{yysec6.9}. The statement
in the abstract is an immediate consequence of the Main Theorem.

\begin{corollary}
Let $k$ be an algebraically closed field of characteristic zero. Let
$H$ be an affine prime regular Hopf algebra of Gelfand-Kirillov
dimension one. If the PI-degree of $H$ is $1$ or prime, then $H$ is
isomorphic to one of the examples listed in \S \ref{yysec3}.
\end{corollary}

\begin{proof} Since $n:=io(H)$ is equal to the PI-degree of $H$,
$n$ is either $1$ or prime. Since $im(H)$ is a divisor of $io(H)$,
$im(H)$ must be either $1$ or $n$. The result follows from the
Main Theorem.
\end{proof}

\section{Questions and conjectures} \label{yysec7}
\subsection{}
\label{yysec7.1} The most obvious question arising from
this paper, one which we hope to address in the future, is:

\begin{question} Is Theorem \ref{yysec0.5} valid without the
hypothesis that $im(H)$ is 1 or $io(H)$?
\end{question}

\subsection{The hypotheses - prime versus semiprime}
\label{yysec7.2}
Suppose $H$ is a \emph{semiprime} regular Hopf algebra of GK-dimension 1
over an algebraically closed field $k$ of characteristic 0. By
\cite[Theorem A]{BG1} $H$ is a finite direct sum of prime algebras,
each of them having GK-dimension 1 because of the Cohen-Macaulay
property. It thus seems reasonable to ask:

\begin{question} What are the Hopf algebras $H$ which satisfy
($\mathbf{H}$), but with prime weakened to semiprime?
\end{question}

Of course, it may be easier to tackle this question first in the
case where $im(H)$ is 1 or $io(H).$

Since $H$ is a direct sum of prime rings, there is exactly one
minimal prime ideal of $H$ - let's call it $P$ - contained in
$\ker \epsilon .$ It is shown in \cite[Theorem 6.5 and Lemma
5.5]{LWZ} that $\overline{H} := H/P$ is a Hopf algebra, and indeed
satisfies all the hypotheses ($\mathbf{H}$). Moreover $io(H) =
io(\overline{H})$ and $im(H) = im(\overline{H}).$ It's conjectured
in \cite[remark 6.6]{LWZ} that there is a short exact sequence of
Hopf algebras
$$ 0 \longrightarrow K \longrightarrow H \longrightarrow
\overline{H} \longrightarrow 0. $$ Here, $K= H^{co \overline{H}},$
the coinvariant subalgebra, and $K$ is conjecturally a normal finite
dimensional Hopf subalgebra of $H$; if this is true, then it's easy
to see that $K$ must be semisimple artinian.

\subsection{The hypotheses of the main theorem}
\label{yysec7.3}
(i) \textbf{The characteristic of} $k$: The assumption
that $k$ has characteristic 0 is needed to ensure that $io(H)$ is
a unit. This hypothesis is not a consequence of the others in
(\textbf{H}), and without it the theorem is false, as is shown by
the example in Remarks (\ref{yysec2.3}) B(a). We can thus ask:

\begin{question} What are the noetherian prime affine Hopf $k$-algebras $H$
over an algebraically closed field $k$, with $\mathrm{GKdim} H =
\mathrm{gldim} H = 1$?
\end{question}

(ii) \textbf{The algebraic closure of }$k$: When $k$ is not
algebraically closed there are in general other 1-dimensional
commutative Hopf $k$-domains, besides $k[x]$ and $k[x^{\pm 1}],$
\cite[8.3]{LWZ}. Therefore we ask:

\begin{question} How should the main theorem be amended if $k$ is not
algebraically closed?\end{question}

(iii) \textbf{Regularity of }$H$: The non-regular generalised Liu
algebras found in (\ref{yysec3.4}) demonstrate that the remaining
hypotheses in (\textbf{H}) do not imply that $H$ has finite global
dimension. Here is another example.

\begin{example}
Let $n, p_0, p_1, \cdots, p_2$ be positive integers and 
$q\in k^{\times}$ with the following properties:
\begin{enumerate}
\item
$s\geq 2$ and $1<p_1 < p_2 < \cdots < p_s$;
\item
$p_0\mid n$ and $p_0, p_1,\cdots, p_s$ are pairwise relatively
prime;
\item
$q$ is a primitive $\ell$-th root of unity, where
$\ell= (n/p_0) p_1 p_2 \cdots p_s$.
\end{enumerate}

Set $m_i=p_i^{-1} \prod_{j=1}^s p_j$ for $i=1, \cdots, s$. Let
$A$ be the subalgebra of $k[y]$ generated by $y_i:= y^{m_i}$
for $i=1,\cdots, s$. The $k$-algebra automorphism 
of $k[y]$ sending $y\mapsto qy$ restricts to an algebra 
automorphism of $A$. There is a unique Hopf algebra 
structure on the skew Laurent polynomial ring 
$B=A[x^{\pm 1}; \sigma]$ such that $x$ is group-like and 
the $y_i$ are skew primitive, with 
$$\Delta(y_i)=y_i\otimes 1+x^{m_i n}\otimes y_i$$
for $i=1,\cdots, s$. It was checked in \cite[Construction 1.2]{GZ}
that this is a Hopf PI domain of GK-dimension two, and is denoted 
by $B(n,p_0,p_1, \cdots,p_s, q)$.

The Hopf algebra $B$ has infinite global dimension, therefore 
provides a negative answer to \cite[Question K]{Bro} 
(see \cite[Remark 1.6]{GZ}).

Now let $H$ be the quotient algebra $B/(x^{\ell}-1)$ where 
$\ell= (n/p_0) p_1 p_2 \cdots p_s$. Since group-like elements 
$x^{\pm \ell}$ generate a central Hopf subalgebra of $B$, 
$(x^{\ell}-1)$ is a Hopf ideal and hence $H$ is a 
quotient Hopf algebra of $B$. 

We claim that $H$ is prime of Gelfand-Kirillov dimension one, but 
not regular. As an algebra $H$ is a 
skew group ring $A\ast {\mathbb Z}/(\ell)$, whence is semiprime 
by \cite[Proposition 7.4.8(1)]{Mo}. Since $\ell$ is the order of $q$, 
the center of $H$ is $k[y^\ell]$ which is a domain. Hence
$H$ is prime. Since $A$ is not regular and $A\ast Z/(\ell)$ is
strongly ${\mathbb Z}/(\ell)$-graded, $A\ast {\mathbb Z}/(\ell)$ is  
not regular. Finally $\GKdim H=\GKdim A=1$.
\end{example}

Up to date, we do not know other examples of non-regular prime
Hopf algebras of GK-dimension one. Hence we should ask:

\begin{question} What other Hopf algebras have to be included if the
regularity hypothesis is dropped from the main theorem?
\end{question}

\subsection{GK-dimension greater than one} 
\label{GK} Beyond
GK-dimension one, there is no doubt that many classes of examples
remain to be discovered. Of course not every noetherian Hopf algebra
of GK-dimension 2 satisfies a polynomial identity: consider the
enveloping algebra of the 2-dimensional non-abelian Lie algebra,
over a field of characteristic 0. A few new examples of 
noetherian Hopf algebras are given in \cite{GZ}. However, simply 
because we don't know many examples, we ask:

\begin{question} Does every noetherian Hopf $k-$algebra of GK-dimension
2 satisfy a polynomial identity if $k$ has positive characteristic?
\end{question}

In arbitrary characteristic, one could start by considering the
Hopf algebra case of Small's famous question

\begin{question} Is every noetherian prime Hopf $k-$algebra of GK-dimension
2 either PI or primitive?
\end{question}

Simply for completeness and to emphasis the fact that much
remains to be understood about the classical components, we
restate the question from Remark (\ref{yysec2.3})B(b):

\begin{question} If $A$ is an AS-Gorenstein noetherian Hopf algebra
with a bijective antipode, is the center of $A$ contained in its
classical component?
\end{question}

\section*{Acknowledgments}
Both authors are supported in part by Leverhulme Research
Interchange Grant F/00158/X (UK). J.J. Zhang is supported by a the
US National Science Foundation. We thank S. Gelaki for helpful
advice.

\providecommand{\bysame}{\leavevmode\hbox to3em{\hrulefill}\thinspace}
\providecommand{\MR}{\relax\ifhmode\unskip\space\fi MR }
\providecommand{\MRhref}[2]{%

\href{http://www.ams.org/mathscinet-getitem?mr=#1}{#2} }
\providecommand{\href}[2]{#2}

\end{document}